\documentclass[11pt]{amsart}
\newcommand{\ds}{\displaystyle}
\newcommand{\ben}{\begin{enumerate}}
\newcommand{\een}{\end{enumerate}}
\newcommand{\eq}[2][label]{\begin{equation}\label{#1}#2\end{equation}}
\newcommand{\av}[2]{\langle #1\rangle_{_{\scriptstyle #2}}}

\usepackage{amsfonts}
\usepackage{amssymb}
\usepackage{amsmath}
\usepackage{graphicx,color,hyperref}

\usepackage{euscript}

\newcommand{\bel}[1]{\boldsymbol{#1}}
\newcommand{\df}{\stackrel{\mathrm{def}}{=}}

\newcommand{\BMO}{{\rm BMO}}

\newcommand{\AQ}{A_2^{d,Q}}

\newcommand{\A}{\mathcal{A}}

\newtheorem{theorem}{Theorem}[section]

\newtheorem{lemma}[theorem]{Lemma}
\newtheorem{corollary}[theorem]{Corollary}

\newtheorem*{theorem*}{Theorem}{\bf}{\it}
\newtheorem*{proposition*}{Proposition}{\bf}{\it}
\newtheorem*{observation*}{Observation}{\bf}{\it}
\newtheorem*{lemma*}{Lemma}{\bf}{\it}

\theoremstyle{definition}

\theoremstyle{remark}
\newtheorem{remark}[theorem]{Remark}

\numberwithin{equation}{section}

\allowdisplaybreaks
\begin{document}

\title{Best constants for a family of Carleson sequences}

\author{Leonid Slavin}
\address{University of Cincinnati}
\email{leonid.slavin@uc.edu}

\thanks{The author's research is supported in part by the NSF (DMS-1001567)}



\date{July 16, 2014}

\begin{abstract}
We consider a general family of Carleson sequences associated with dyadic $A_2$ weights and find sharp -- or, in one case, simply best known -- upper and lower bounds for their Carleson norms in terms of the $A_2$-characteristic of the weight. The results obtained make precise and significantly generalize earlier estimates by Wittwer, Vasyunin, Beznosova,  and others. We also record several corollaries, one of which is a range of new characterizations of dyadic $A_2.$ Particular emphasis is placed on the relationship between sharp constants and optimizing sequences of weights; in most cases explicit optimizers are constructed. Our main estimates arise as consequences of the exact expressions, or explicit bounds, for the Bellman functions for the problem, and the paper contains a measure of Bellman-function innovation. 
\end{abstract}
\maketitle

\section{Preliminaries}
We will be concerned with weights on $\mathbb{R},$ i.e. locally integrable functions that are positive almost everywhere. Our weights will be assumed to belong to the dyadic Muckenhoupt class $A^d_2$ associated with a particular lattice $\mathcal{D},$ i.e., the set of all dyadic intervals on the line uniquely determined by the position and size of the root interval. 
If an interval $I$ is fixed, $\mathcal{D}(I)$ will stand for the unique dyadic lattice on $I$ and $\mathcal{D}_n(I)$ for the set of the dyadic subintervals of $I$ of the $n$-th generation, $\mathcal{D}_n(I)=\{J:J\in\mathcal{D}(I), |J|=2^{-n}|I|\}.$ 

Let $\av{w}J$ be the average of a weight $w$ over an interval $J,$ $\av{w}J=\frac1{|J|}\int_Jw.$ A weight $w$ is said to belong to $A^d_2$, written $w\in A^d_2,$ if
$$
[w]_{A^d_2}\df\sup_{J\in\mathcal{D}}\av{w}J\av{w^{-1}}J<\infty.
$$
The quantity $[w]_{A^d_2}$ is referred to as the $A_2^d$-characteristic of $w.$ Observe that $[w]_{A^d_2}\geqslant1$ and $w\in A_2^d$ if and only if $w^{-1}\in A_2^d,$ in which case $[w]_{A^d_2}=[w^{-1}]_{A^d_2}.$ For a number $Q\geqslant 1,$ the set of all $A_2^d$ weights $w$ with $[w]_{A^d_2}\leqslant Q$ will be denoted by $\AQ.$ If $I$ is an interval and the supremum in the above definition is taken over all $J\in\mathcal{D}(I)$ instead of all $J\in\mathcal{D},$ we will write $A_2^d(I)$ and $\AQ(I),$ as appropriate.

A non-negative sequence $\{c_J\}_{J\in\mathcal{D}}$ is called a Carleson sequence if for all $I\in\mathcal{D}$
\eq[2]{
\frac1{|I|}\sum_{J\in\mathcal{D}(I)} c_J\leqslant C<\infty.
}
The smallest such $C$, denoted by $\|\{c_J\}\|_{\mathcal{C}},$ is called the Carleson norm of $\{c_J\}.$
The significance of such sequences in analysis stems mainly from their role in the Carleson embedding theorem and related results. A key example is the following lemma whose proof can be found in~\cite{MP} (in a more general, weighted setting). We will use this lemma to derive an important corollary of our main estimates. 
\begin{lemma}[Carleson Lemma]
\label{carleson}
A sequence $\{c_J\}_{J\in\mathcal{D}}$ is a Carleson sequence with norm $B$ if and only if for all non-negative, measurable functions $F$ on the line,
$$
\sum_{J\in\mathcal{D}}c_J\,\inf_{x\in J}F(x)\leqslant B\int_{\mathbb{R}}F(x)\,dx.
$$
\end{lemma}

This inequality suggests that it may be important to have good estimates of the Carleson norms of sequences one uses. One specific context where such need arises is when studying dyadic paraproducts in weighted settings. For example, in \cite{Bez} and \cite{MP} the authors estimate the norms of paraproducts on $L^2(w),$ with $w\in A^d_2$. In this situation, the sequences of interest are often of the form
$$
c^{(\alpha)}_J(w)\df|J|\,\av{w}J^\alpha\av{w^{-1}}J^\alpha\,\left[\frac{(\Delta_Jw)^2}{\av{w}J^2}+\frac{(\Delta_Jw^{-1})^2}{\av{w^{-1}}J^2}\right],
$$
where $\Delta_J(\cdot)=\av{ \cdot}{J^-}\!\!-\av{ \cdot}{J^+},$ and $J^\pm$ are the two halves of $J.$ 
In \cite{Bez}, Beznosova showed that the norm of $\{c^{(1/4)}_J(w)\}$ is no greater than $C [w]_{A^d_2}^{1/4}$ for a numerical constant $C.$ In \cite{NV}, Nazarov and Volberg extended this estimate, proving that 
\eq[0.1]{
\|\{c^{(\alpha)}_J(w)\}\|_{\mathcal{C}}\leqslant\frac{C}{\alpha-2\alpha^2}[w]_{A^d_2}^{\alpha},\quad \alpha\in(0,1/2).
}
In a recent paper \cite{MP}, Moraes and Pereyra observed that one can simply combine Beznosova's result with the fact that $t\mapsto t^\alpha$ is an increasing function for $\alpha>0$ to obtain
$$
\|\{c^{(\alpha)}_J(w)\}\|_{\mathcal{C}}\leqslant r(\alpha)[w]_{A^d_2}^{\alpha},\quad \alpha>0,
$$
for an explicit $r(\alpha).$

The case $\alpha=0$ was considered by Wittwer in \cite{Wit}, where she obtained the asymptotically sharp estimate 
\eq[0.2]{
\|\{a_J(w)\}\|_{\mathcal{C}}\leqslant8\log[w]_{A_2^d},
}
where $\{a_J(w)\}$ is the same as $\{c^{(0)}_J(w)\},$ except with only {\it one} of the two summands (it does not matter which one is removed). One can double Wittwer's constant to get an estimate for $\{c_J^{(0)}(w)\},$ but as we will see below, it is no longer sharp, though the logarithm remains. The same sequence $\{a_J(w)\},$ but for $w\in A_\infty,$ was studied by Vasyunin in~\cite{vas}. He found the upper and lower Bellman functions for the problem, and constructed explicit optimizing sequences. 

Motivated by these results, we consider general sequences $\{c_J^{\Phi}(w)\}$, defined for $w\in A_2^d$ and a non-negative function~$\Phi$ on $[1,\infty)$ by
\eq[3g]{
c^{\Phi}_J(w)=|J|\,\Phi\big(\av{w}J\av{w^{-1}}J\big)\,\left[\frac{(\Delta_Jw)^2}{\av{w}J^2}+\frac{(\Delta_Jw^{-1})^2}{\av{w^{-1}}J^2}\right].
}
Our ultimate goal is to obtain sharp estimates on $\|\{c^\Phi_J(w)\}\|_{\mathcal{C}}$ in terms of $[w]_{A^d_2},$ i.e. find the largest function $k_\Phi$ and the smallest function $K_\Phi$ in the following inequality: 
$$
k_{\Phi}\big([w]_{A^d_2}\big)\leqslant \sup_{I\in D}\frac1{|I|}\sum_{J\in D(I)}c_J^\Phi(w) \leqslant K_{\Phi}\big([w]_{A^d_2}\big).
$$
Let us keep the names $k_\Phi$ and $K_\Phi$ for these best functions: for all $Q\geqslant1,$ we define
\eq[k_K]{
K_\Phi(Q)=\sup_{w:[w]_{A_2^d}= Q}\|\{c^\Phi_J(w)\}\|_{\mathcal{C}},\quad k_\Phi(Q)=\inf_{w:[w]_{A_2^d}= Q}\|\{c^\Phi_J(w)\}\|_{\mathcal{C}}.
}

In the special case $\Phi(t)=t^\alpha,$ we will use the notation $K_\alpha$ and $k_\alpha,$ respectively. Under some fairly mild conditions on $\Phi,$ we derive explicit formulas for $K_\Phi$ and $k_\Phi,$ except for one class of $\Phi$ where we simply find good upper estimates. In particular, we compute $k_\alpha$ for all $\alpha\in(-\infty,\infty)$ and $K_\alpha$ for all $\alpha\in(-\infty,\infty)\setminus(1,\frac32).$ (This seemingly peculiar exclusion is explained in the next section; for $\alpha\in(1,\frac32)$ we provide an estimate on $K_\alpha.$)

We deduce our formulas for $k_\Phi$ and $K_\Phi$ from explicit expressions, or, for some $\Phi,$ explicit estimates, for the upper and lower {\it Bellman functions} for the dyadic sum
$$
\frac1{|I|}\sum_{J\in\mathcal{D}}c^\Phi_J(w).
$$
These functions, formally defined in the next section, are simply the supremum and infimum of this sum, taken over all $w\in A_2^d(I)$ with certain parameters fixed. Once in hand, a Bellman function yields not only a continuum of sharp constants (one for each choice of the parameters), but also the optimizing sequences of weights that realize those constants in the limit.

We continue this discussion in the next section, after key objects have been introduced and the main results stated.

\section*{Acknowledgment}
The author is grateful to Vasily Vasyunin for valuable discussions related to the project.

\section{Main results, corollaries, and discussion}

Let us set  notation for the rest of the paper. Alongside $\Phi,$ it will be convenient to use the following two functions, also defined on $[1,\infty):$ 
\eq[fh]{
f(s)=\Phi(s^2),\qquad h(s)=\frac{f(s)}{s^2}.
}

In addition, if the parameter $Q\geqslant1$ is fixed, $w$ is a weight, and $J\in\mathcal{D},$ we let 
\eq[sr]{
L=\sqrt{Q},\qquad s_J(w)=\sqrt{\av{w}J\av{w^{-1}}J},\qquad R_J(w)=\frac{(\Delta_Jw)^2}{\av{w}J^2}+\frac{(\Delta_Jw^{-1})^2}{\av{w^{-1}}J^2}.
}
In this notation~\eqref{3g} becomes
$$
c_J^\Phi(w)=|J|\,f\big(s_J(w)\big)\,R_J(w).
$$

Our headline results are the following two theorems. 
\begin{theorem}
\label{ttt1}
~
\ben
\item
If $\Phi$ is increasing and $h$ is convex, then
$$
K_{\Phi}(Q)=16\Phi(Q)\Big(1-\frac1{\sqrt Q}\Big) +8\int_1^Q \frac{\Phi(y)}y\Big(1-\frac1{\sqrt y}\Big)\,dy.
$$
\item
\label{p2}
If $\Phi$ is increasing and $h$ is concave, then
$$
K_{\Phi}(Q)\leqslant 
8(\sqrt{Q}-1)(3\sqrt{Q}-1)\,h(s_0(Q)),
$$
where
$$
s_0(Q)=\frac{8Q-\sqrt{Q}-1}{3(3\sqrt{Q}-1)}.
$$

\item
If $\Phi$ is decreasing, then
$$
K_{\Phi}(Q)=8\int_1^Q\frac{\Phi(y)}{y}\,dy.
$$
\een
\end{theorem}
\begin{remark}
As the formulation above suggests, we do not know the actual function $K_\Phi$ in the case when $h$ is concave. That is because the solution of the corresponding extremal problem differs in structure from that for other $\Phi;$ in particular, it does not arise as the solution of a differential equation. Instead, we provide an upper estimate. As explained in Section~\ref{concave}, the estimate given above can be slightly improved, but the result is far less transparent and likely still not sharp. The estimate in Part~(\ref{p2}) is the best known and easy to use, especially when $\Phi$ is specified to be a power function (see below).
\end{remark}
Our second theorem concerns lower estimates.

\begin{theorem}
\label{ttt2}
~
\ben
\item
If $h$ is increasing, then
$$
k_\Phi(Q)=8\int_1^Q\frac{\Phi(y)}{y}\,dy.
$$
\item
If $h$ is decreasing, then
$$
k_\Phi(Q)=8\Phi(Q)\Big(1-\frac1Q\Big).
$$
\een
\end{theorem}
Note that we make no explicit assumptions of pointwise differentiability, or even continuity, on the part of $\Phi$ or $h.$ Let us restate these results for the case $\Phi(t)=t^\alpha.$ 
\begin{corollary}
\label{ttt3}
$$
\begin{array}{ll}
K_\alpha(Q)&=
\begin{cases}
\ds\frac{8(2\alpha+1)}\alpha\,Q^\alpha-\frac{32\alpha}{2\alpha-1}\,Q^{\alpha-1/2}+\frac8{\alpha(2\alpha-1)},&~\alpha\in\big(0,\textstyle{\frac12}\big)\cup\big(\textstyle{\frac12},1\big]\cup\big[\textstyle{\frac32},\infty\big);
\vspace{2mm}

\\
32Q^{1/2}-8\log Q-32,&~\alpha=1/2;
\vspace{2mm}

\\
\ds8\log Q,&~\alpha=0;
\vspace{2mm}

\\
\ds\frac8\alpha(Q^\alpha-1),&~\alpha<0;
\end{cases}\\
\\
K_\alpha(Q)&\leqslant 
8\cdot 3^{2-2\alpha}\,(\sqrt Q-1)(8Q-\sqrt Q-1)^{2\alpha-2}(3\sqrt Q-1)^{3-2\alpha}
\vspace{2mm}

\\
&=\big[K_{1}(Q)\big]^{3-2\alpha}\,\big[K_{3/2}(Q)\big]^{2\alpha-2},\hspace{3.5cm}\alpha\in\big(1,\textstyle{\frac32}\big);
\\
\\
k_\alpha(Q)&=
\begin{cases}
\ds\frac8\alpha(Q^\alpha-1),&\hspace{5cm}\alpha\geqslant 1;
\vspace{2mm}

\\
\ds 8\,\big(Q^\alpha-Q^{\alpha-1}\big),&\hspace{5cm}\alpha<1.
\end{cases}
\end{array}
$$
\end{corollary} 
\begin{remark}
Thus, for $\alpha>0, \alpha\notin(1,3/2),$ the sharp order of growth is $K_\alpha(Q)\approx 8(2+1/\alpha)Q^\alpha$  for large $Q.$ On the other hand, for $\alpha\in(1,3/2)$ using the formula from Part~(\ref{p2}) of Theorem~\ref{ttt1} amounts simply to log-linear interpolation between the sharp results for $\alpha=1$ and $\alpha=3/2,$ or, equivalently, to an application of H\"older's inequality. The logarithmic terms in the cases $\alpha=0$ and $\alpha=1/2$ correspond to the blow-up of the Nazarov--Volberg estimate~\eqref{0.1} for these values of $\alpha.$ Lastly, note that we get the same constant for $\|\{c^{(0)}_J(w)\}\|_{\mathcal{C}}$ as Wittwer did for $\|\{a_J(w)\}\|_{\mathcal{C}}$ in~\eqref{0.2}, even though $c^{(0)}_J(w)=a_J(w)+a_J(w^{-1}).$ This reflects the fact that the quantities $a_J(w)$ and $a_J(w^{-1})$ cannot be too large at the same time.

\end{remark}

Let us record several related results, all of which are proved in the next section.
First, the lower estimates of Theorem~\ref{ttt2} allow us to obtain a range of equivalent definitions of $A_2^d$ in terms of sequences $\{c_J^\Phi\}$ for all increasing $\Phi$ such that $\Phi(t)\to\infty$ as $t\to\infty.$ We note that this limit condition, as well as the inverted lower estimate in Part~(ii) have direct analogs in a recent paper~\cite{LSSVZ} concerning equivalent definitions of $\BMO.$ 

\begin{theorem}
\label{corr1}
Assume that $\Phi$ is increasing. 
\ben
\item[(i)] If $w\in A_2^d,$ then the sequence $\{c^\Phi_J(w)\}$ is Carleson and
$$
\|\{c^{\Phi}_J(w)\}\|_{\mathcal{C}}\leqslant 8\Phi\big([w]_{A_2^d}\big) \log\big([w]_{A_2^d}\big).
$$
\item[(ii)] 
\label{pii}
If $\{c^\Phi_J(w)\}$ is Carleson and $\ds\lim_{t\to\infty}\Phi(t)=\infty,$ then $w\in A_2^d$ and
$$
[w]_{A_2^d}\leqslant v\big(\|\{c^{\Phi}_J(w)\}\|_{\mathcal{C}}\big),
$$
where $v: [0,\infty)\to [1,\infty)$ is the inverse to the function $\ds u(Q)=\frac8Q\int_1^Q\Phi(t)\,dt, Q\geqslant1.$
\item[(iii)]
If $\ds\lim_{t\to\infty}\Phi(t)\ne\infty,$ then there exists a weight $w\notin A^d_2$ such that $\{c^{\Phi}_J(w)\}$ is Carleson.
\een
\end{theorem}

The main thrust of this theorem is that the condition $\lim_{t\to\infty}\Phi(t)=\infty$ is necessary and sufficient for the implication ``$\{c^\Phi_J(w)\}$ is Carleson'' $\Rightarrow w\in A_2^d.$ The quantitative estimates are stated so that they work for all increasing $\Phi;$ as such, they are sharp only on the class of all such $\Phi.$ If one has a specific $\Phi$ that falls under one of the cases in both Theorem~\ref{ttt1} and Theorem~\ref{ttt2}, one can, in general, do better by using the estimates from those theorems. In particular, we have the following corollary.

\begin{corollary}
\label{corr2}
~
\ben
\item[(i)]
If $\alpha>0$ and $\{c^{(\alpha)}_J(w)\}$ is a Carleson sequence and, then $w\in A^d_2$ and
$$
[w]_{A^d_2}\leqslant k_\alpha^{-1}\big(\|\{c^{(\alpha)}_J(w)\}\|_{\mathcal{C}}\big),
$$
where $k^{-1}_\alpha: [0,\infty)\to [1,\infty)$ is the inverse to $k_\alpha$ given in Theorem~\ref{ttt4}. This estimate is sharp.
\item[(ii)]
If $\alpha\leqslant 0,$ there exists a weight $w\notin A^d_2$ such that $\{c^{(\alpha)}_J(w)\}$ is a Carleson sequence.
\een
\end{corollary}

The reader will notice that the sequence $\{c_J^\Phi(w)\}$ does not change if we replace $w$ with $\tau w$ for $\tau>0.$ This zero-degree homogeneity is of central importance in our sharp proofs, but if we are willing to slightly compromise on sharpness, we can easily extend our main theorems to settings with different homogeneity. To illustrate this point, let us examine the sequence
$$
c^{(\alpha,\beta)}_J(w)=|J|\,\av{w}J^\alpha\,\av{w^{-1}}J^\beta\, R_J(w)
$$ 
for $\alpha\ne\beta.$ For this sequence, we can obtain a range of inequalities of the flavor studied in~\cite{BR}.
\begin{corollary}
\label{corr3}
For any $w\in A_2^d$ and any $\alpha,\beta$ such that $\alpha>\beta,$ 
$$
\frac1{|I|}\sum_{J\in\mathcal{D}(I)}c^{(\alpha,\beta)}_J(w)\leqslant e\,K_\alpha\big([w]_{A_2^d}\big)\,\av{w^{\alpha-\beta}}I.
$$
\end{corollary}
\begin{remark}
Of course, the same conclusion holds if we interchange $w$ and $w^{-1}.$ In addition, if $\alpha-\beta\leqslant1,$ we can replace $\av{w^{\alpha-\beta}}I$ with $\av{w}I^{\alpha-\beta}$ by H\"older's inequality. 
\end{remark}

We now return to Theorems~\ref{ttt1} and~\ref{ttt2}. To explain where they come from, we need to define the Bellman functions for our problem. Fix $Q\geqslant1$ and let
$$
\Omega_Q=\{(x_1,x_2):1\leqslant x_1x_2\leqslant Q\}.
$$ 
For each $x=(x_1,x_2)\in\Omega_Q$ and each $I\in\mathcal{D},$ let 
$$
E_{Q,x,I}=\big\{w:~w\in\AQ(I),~\av{w}I=x_1,~\av{w^{-1}}I=x_2\big\}.
$$
The upper and lower Bellman functions are defined, respectively, by
\eq[6.6n]{
\bel{B}_{Q,\Phi}(x_1,x_2)=\sup\Big\{
\frac1{|I|}\sum_{J\in\mathcal{D}(I)}c^{\Phi}_J(w):~w\in E_{Q,x,I}
\Big\}
} 
and 
\eq[6.7n]{
\bel{b}_{Q,\Phi}(x_1,x_2)=\inf\Big\{
\frac1{|I|}\sum_{J\in\mathcal{D}(I)}c^{\Phi}_J(w):~w\in E_{Q,x,I}
\Big\}.
} 
If $\Phi$ is a power function, $\Phi(t)=t^\alpha,$ we will write $\bel{B}_{Q,\alpha}$ and $\bel{b}_{Q,\alpha}$ for $\bel{B}_{Q,\Phi}$ and $\bel{b}_{Q,\Phi},$ respectively. 

One immediately observes that the functions $\bel{B}_{Q,\Phi}$ and $\bel{b}_{Q,\Phi}$ are independent of the interval $I$ that formally enters into their definitions. In addition, for any $w\in\AQ$ and any $I\in\mathcal{D},$  we have
$(\av{w}I,\av{w^{-1}}I)\in\Omega_Q,$ thus, the functions $\bel{B}_{Q,\Phi}$ and $\bel{b}_{Q,\Phi}$ are defined, at least formally, on $\Omega_Q.$ The fact that $E_{Q,x,I}$ is nonempty for every $Q\geqslant1,$ every interval $I\in\mathcal{D},$ and every $x\in\Omega_Q$ is subsumed in the statements of Theorems~\ref{mbu} and~\ref{mbl} below.

For a number $L\geqslant1$ and a non-negative function $f$ on $[1,\infty)$ formally define:
\eq[701]{
A_{L,f}(s)=16\left[\frac{f(L)}L+\int_1^L\frac{f(z)}{z^2}\,dz\right](s-1)-16\int_1^s\frac{f(z)}{z^2}\,(s-z)\,dz,
}
\eq[702]{
a_f(s)=16\int_1^s\frac{f(z)}z\,dz.
}
For reasons that will soon be made clear, we will often refer to these two functions as {\it Bellman candidates}. We are now in a position to state the main theorems about the functions $\bel{B}_{Q,\Phi}$ and $\bel{b}_{Q,\Phi},$ which are used in the next section to prove Theorems~\ref{ttt1} and~\ref{ttt2}.
\begin{theorem}
\label{mbu}
In the notation~\eqref{fh} and~\eqref{sr},
\ben
\item
If $\Phi$ is increasing and $h$ is convex, then
$$
\bel{B}_{Q,\Phi}(x_1,x_2)=A_{L,f}(\sqrt{x_1x_2}).
$$
\item
If $\Phi$ is increasing and $h$ is concave, then
$$
\bel{B}_{Q,\Phi}(x_1,x_2)\leqslant 8(\sqrt{x_1x_2}-1)\big(4L-\sqrt{x_1x_2}-1\big)\,h\big(s_0(\sqrt{x_1x_2})\big),
$$
where
$$
s_0(s)=\frac{9L^2-s^2-s-1}{3(4L-s-1)}.
$$

\item
If $\Phi$ is decreasing, then
$$
\bel{B}_{Q,\Phi}(x_1,x_2)=a_f(\sqrt{x_1x_2}).
$$
\een
\end{theorem}

\begin{theorem}
\label{mbl}
In the notation \eqref{fh} and~\eqref{sr},
\ben
\item
If $h$ is increasing, then
$$
\bel{b}_{Q,\Phi}(x_1,x_2)=a_f(\sqrt{x_1x_2}).
$$
\item
If $h$ is decreasing, then
$$
\bel{b}_{Q,\Phi}(x_1,Q/x_1)=8\,\Phi(Q)\,\Big(1-\frac1Q\Big).
$$
\een
\end{theorem}
\begin{remark}
Note that in all cases in these theorems, all integrals involved in the definitions of $A_{L,f}$ and $a_f$ are well-defined.

To summarize: in Parts (1) and (3) of Theorem~\ref{mbu} and in Part~(1) of Theorem~\ref{mbl}, we find the exact Bellman functions at every point of $\Omega_Q;$ in Part~(2) of Theorem~\ref{mbu}, we provide a point-wise majorant for $\bel{B}_{Q,\Phi};$ and in Part~(2) of Theorem~\ref{mbl}, we find $\bel{b}_{Q,\Phi}$ only on the boundary curve $x_1x_2=Q$ (which, however, is enough to compute $k_\Phi$ for this case).

\end{remark}

For $\Phi(t)=t^\alpha,$ these formulas become
\begin{corollary}
\label{ttt4}
Let $s=\sqrt{x_1x_2}.$
$$
\begin{array}{ll}
\bel{B}_{Q,\alpha}(x_1,x_2)&=
\begin{cases}
\ds\frac8{\alpha(2\alpha-1)}(1-s^{2\alpha})+\frac{32\alpha}{2\alpha-1}L^{2\alpha-1}(s-1),&
~\alpha\in\big(0,\textstyle{\frac12}\big)\cup\big(\textstyle{\frac12},1\big]\cup\big[\textstyle{\frac32},\infty\big);
\vspace{2mm}

\\
32(s-1)-16\log L+16\,s\,\log\big(\frac Ls\big),&~\alpha=1/2;
\vspace{2mm}

\\
\ds 16\log s,&~\alpha=0;
\vspace{2mm}

\\
\ds\frac8\alpha(s^{2\alpha}-1),&~\alpha<0;
\end{cases}\\
\\
\bel{B}_{Q,\alpha}(x_1,x_2)&\leqslant 
\big[\bel{B}_{Q,1}(x_1,x_2)\big]^{3-2\alpha}\,\big[\bel{B}_{Q,3/2}(x_1,x_2)\big]^{2\alpha-2},\hspace{1.2cm}\alpha\in\big(1,\textstyle{\frac32}\big);
\\
\\
\bel{b}_{Q,\alpha}(x_1,x_2)&=

\ds\frac8\alpha(s^{2\alpha}-1),\hspace{6cm}\alpha\geqslant 1;
\vspace{2mm}

\\
\bel{b}_{Q,\alpha}(x_1,Q/x_1)&=\ds 8\,L^{2\alpha-2}\big(L^2-1\big),\hspace{5.2cm}\alpha<1.
\end{array}
$$
\end{corollary} 

Before proceeding with the proofs of these theorems and corollaries, we briefly discuss the principal method at play and make several related remarks.

\subsection{The method}
The Bellman function technique in analysis consists of finding special functions with designated size and convexity properties, to aid in induction-based proofs. In martingale settings, such as ours, the first Bellman-function proofs can be found in the work of Burkholder~\cite{b}. Later, the technique was applied, in a significantly different form and under the current name, to many problems in harmonic analysis by Nazarov, Treil, and Volberg \cite{ntv1,nt,ntv2}, followed by others; the reader can consult the lecture notes~\cite{vol} for details. We also note the work of Os\c{e}kowski on sharp inequalities for martingales~\cite{ose1,ose2}, which implements what can be termed a mixed Burkholder--Bellman method.

An important distinction exists between two kinds of Bellman functions. One kind is the {\it true} Bellman functions, which are defined as solutions of extremal problems such as~\eqref{6.6n} and~\eqref{6.7n}. They provide complete information about the inequality in question, including the structure of optimizers, but may be difficult to compute. The other and much more prevalent kind are {\it Bellman-type} functions, which are useful substitutes with properties similar to those of the true function. These are non-unique and much easier to find; however, they rarely produce completely sharp estimates and never the exact optimizers.
Inequality~\eqref{0.1}, due to Beznosova and  Nazarov--Treil, is a good example of the use of Bellman-type functions. Observe that the bound obtained is sharp in $[w]_{A_2^d},$ but sub-optimal in the numerical factor.

Recent advances in finding true Bellman functions have made them viable as stand-alone objects of study, ahead of the inequalities they imply; that perspective is maintained here. To find the true function, one often needs to solve a PDE (a notable exception is the work of Melas and co-authors~\cite{melas1,melas2,mn}, where one determines the function starting from external information about the optimizers). The functions~$A_{L,f}$ and $a_f$ from~\eqref{701} and~\eqref{702} are two homogenized solutions of such a PDE. Overall, the Bellman-related computations in this paper are somewhat similar to those in Vasyunin's proof of Buckley's inequality~\cite{vas}, though they are much more general. We also incorporate elements from~\cite{ssv} and~\cite{vv}, particularly in the construction of optimizing sequences.

The reader familiar with the method will find in the paper at least two technical novelties that are likely to be useful in other settings. First, in Section~\ref{main_ineq}, when verifying key inequalities for our Bellman candidates (the so-called ``main inequalities''), we prove statements that are both stronger and simpler than those required. That may seem like a dangerous overreach, as true-Bellman proofs are typically very tight. However, the inequalities we need and the inequalities we actually prove are extremized by the same configurations of the variables involved, and the slack we introduce happens away from those configurations. Second, in Section~\ref{optimizers}, when constructing the optimizer for one of the candidates, we only do so for a special selection of points in the domain $\Omega_Q,$ and then use the {\it a priori} continuity of our Bellman functions to get the desired result. Doing so saves us the messy work of constructing the optimizer for every point of the domain.

\subsection{Extensions}
The results presented above suggest several possible generalizations. The first such generalization concerns computing the functions $k_\Phi$ and $K_\Phi$ for an arbitrary $\Phi.$ The main obstacle, of course, is that we do not yet know $K_\Phi$ when $h(s)=\Phi(s^2)/s^2$ is concave. Section~\ref{concave} discusses a possible approach to computing the Bellman function $\bel{B}_{Q,\Phi}$ (and, hence, $K_\Phi$) for such $\Phi.$ It is plausible that once that solution is in hand, and one thus has three different Bellman candidates corresponding to the three parts of Theorem~\ref{mbu}, one can deal with an arbitrary $\Phi$ by gluing those candidates in an appropriate manner at the points where $\Phi$ and/or $h$ change behavior. 

One can ask whether our results have an analog in higher dimensions. Of course, the definition of $\{c_J^\Phi(w)\}$ has to change:  we need to replace expressions such as $\av{w}{J^-}\!\!\!-\av{w}{J^+},$ where $J$ a dyadic interval, with $\av{w}{J}\!\!-\av{w}{P(J)},$ where $J$ is a dyadic cube and $P(J)$ its unique parent. Beyond that, one must adapt the inductive argument at the center of the proofs of Theorems~\ref{mbu} and~\ref{mbl} to the situation where one must keep track of $2^n$ points from the domain $\Omega_Q$ on every step of the induction, as opposed to just two points before. A somewhat similar challenge, but without regard for sharpness, was successfully handled in~\cite{LSSVZ}.

Lastly, one may be interested in studying the sequence $\{c_J^\Phi(w)\}$ when $w$ is an $A_p$ weight with $p\ne2$ or a reverse H\"older weight. It seems likely that when the definition of $c_J^\Phi$ is properly symmetrized, most of our analysis will go through in such settings, though we have not explored that.

\subsection{Linearity of sharp estimates}
Let us turn things around a bit and consider the function $K_\Phi$ as an operator on $\Phi,$ taking in and returning non-negative functions on $[1,\infty).$ For $t\geqslant1,$ let
$$
(T\Phi) (t)= K_\Phi(t).
$$

Directly from its definition, $T$ is sublinear on the cone of all non-negative $\Phi,$ meaning that for any $a_1,a_2\geqslant0,$
$T(a_1\Phi_1+a_2\Phi_2)\leqslant a_1 T\Phi_1+a_2 T\Phi_2.$
However, Theorem~\ref{ttt1} shows that this operator in fact behaves linearly on broad classes of $\Phi.$ For example, for all increasing $\Phi$ such that $h(s)=\Phi(s^2)/s^2$ is concave, $T\Phi$ is given by the linear (in $\Phi$) expression in Part~(1) of Theorem~\ref{ttt1}. That means that if $\Phi_1, \Phi_2$ are two such functions, then $T(\Phi_1+\Phi_2)=T\Phi_1+T\Phi_2,$ which in terms of the original formulation means that to get a sharp estimate for the Carleson norm of $\{c^{(\Phi_1+\Phi_2)}_J\}$ we need simply add the sharp estimates for $\Phi_1$ and $\Phi_2.$ In particular, if $P$ is a polynomial with non-negative coefficients  $a_0,...,a_n,$ then by Corollary~\ref{ttt3}, $K_P(Q)=\sum_{j=0}^n a_j K_j(Q).$ 

Similar observations can be made about the function $k_\Phi.$ What accounts for this linearity phenomenon is the nature of optimizers, i.e., those sequences of weights on which the supremum and infimum are attained (in the limit) in the definitions~\eqref{6.6n} and~\eqref{6.7n} of the functions $\bel{B}_{Q,\Phi}$ and $\bel{b}_{Q,\Phi}.$ As shown in Section~\ref{optimizers} below, these optimizers do not depend on the exact $\Phi$ involved, but only on the differential structure of the candidates $A_{L,f}$ and $a_f.$

\subsection{Outline}
The rest of the paper is organized as follows: in Section~\ref{elementary}, we prove all results stated in this section, except Theorem~\ref{mbu} and Theorem~\ref{mbl}(1); in Section~\ref{induction}, we establish key properties of the Bellman functions $\bel{B}_{Q,\Phi}$ and $\bel{b}_{Q,\Phi}$ and derive the Bellman candidates $A_{L,f}$ and $a_f;$ in Section~\ref{main_ineq}, we verify that $A_{L,f}$ and $a_f$ bound $\bel{B}_{Q,\Phi}$ and $\bel{b}_{Q,\Phi}$ from above or below, as appropriate; Section~\ref{concave} contains the proof of Theorem~\ref{mbu}(2), as well as a brief discussion of that case; finally, in Section~\ref{optimizers}, we present a detailed construction of optimizers for $A_{L,f}$ and $a_f,$ thus finishing the proofs of Theorems~\ref{mbu} and~\ref{mbl}.

\section{Shorter proofs}
\label{elementary}
The order of proofs is as follows: Theorem~\ref{mbl}(2), Theorem~\ref{ttt1}, Theorem~\ref{ttt2}, Theorem~\ref{corr1}, Corollary~\ref{corr2}, and Corollary~\ref{corr3}. 

\begin{proof}[Proof of Theorem~\ref{mbl}, Part (2)] Take any $Q\geqslant1,$ any interval $I\in\mathcal{D},$ and any $w\in\AQ$ such that 
$s^2_I(w)=\av{w}I\av{w^{-1}}I=Q.$ Since the function $h(s)=f(s)/s^2$ is decreasing,
we have
\begin{align}
\label{key}
\frac1{|I|}\sum_{J\in\mathcal{D}(I)}c_J^\Phi(w)&=\frac1{|I|}\sum_{J\in\mathcal{D}(I)} |J|f(s_J(w))R_J(w)\\
\notag&\geqslant\frac{f(L)}{L^2}\,\frac1{|I|}\sum_{J\in\mathcal{D}(I)} |J|s_J^2(w)\,R_J(w)
\geqslant \frac{f(L)}{L^2}\,\bel{b}_{Q,1}\big(\av{w}I,\av{w^{-1}}I\big),
\end{align}
where $\bel{b}_{Q,1}$ is the Bellman function defined by~\eqref{6.7n} with $\Phi(t)=t$ or, equivalently, $f(s)=s^2.$ For this $f,$ $h(s)=1,$ and we can apply the first part of the theorem: $\bel{b}_{Q,1}(x_1,x_2)=8(x_1x_2-1).$ Therefore, returning to the original $\Phi$ and $f,$ we obtain
$$
\frac1{|I|}\sum_{J\in\mathcal{D}(I)}c_J^\Phi(w)\geqslant 8\,\frac{f(L)}{L^2}\,(L^2-1)=8\,\Phi(Q)\Big(1-\frac1Q\Big),
$$
which means that $\bel{b}_{Q,\Phi}(x_1,Q/x_1)\geqslant 8\,\Phi(Q)\big(1-\frac1Q\big).$ 

To prove equality, take any $I$ and any $x_1>0$ and consider the weight 
\eq[w**]{
w^*(t)=\begin{cases}
x_1\Big(1-\sqrt{1-\frac1Q}\Big),&t\in I^-,
\medskip

\\
x_1\Big(1+\sqrt{1-\frac1Q}\Big),&t\in I^+.
\end{cases}
}
We compute: $\av{w^*}I=x_1,$ $\av{(w^*)^{-1}}I=Q/x_1,$ and $R_I(w^*)=8(1-1/Q);$ also note that $w^*$ is constant on each $J\in\mathcal{D}(I^-)\cup\mathcal{D}(I^+).$ Therefore, $[w^*]_{\AQ(I)}=Q$ and
$$
\bel{b}_{Q,\Phi}(x_1,Q/x_1)\leqslant\frac1{|I|}\sum_{J\in\mathcal{D}(I)}c_J^\Phi(w^*)=\frac1{|I|}\,|I| f(s_I(w^*)) R_I(w^*)=8\Phi(Q)\Big(1-\frac1Q\Big).
$$
\end{proof}

\begin{proof}[Proof of Theorem~\ref{ttt1}]
We will derive this theorem as a corollary of Theorem~\ref{mbu}. Take  $Q\geqslant1$ and a $w\in A^d_2$ such that $Q=[w]_{A^d_2}.$ For any $I\in\mathcal{D},$ by the definition of $\bel{B}_{Q,\Phi},$
$$
\frac1{|I|}\sum_{J\in\mathcal{D}(I)}c_J^\Phi(w)\leqslant \bel{B}_{Q,\Phi}\big(\av{w}I,\av{w^{-1}}I\big)\leqslant \sup_{x\in\Omega_Q}\bel{B}_{Q,\Phi}(x_1,x_2).
$$
Therefore, $\|\{c_J^\Phi(w)\}\|_{\mathcal{C}}\leqslant \sup_{x\in\Omega_Q}\bel{B}_{Q,\Phi}(x_1,x_2)$ and, finally,
$$
K_\Phi(Q)=\sup_{[w]_{A_2^d}=Q} \|\{c_J^\Phi(w)\}\|_{\mathcal{C}}\leqslant\sup_{x\in\Omega_Q}\bel{B}_{Q,\Phi}(x_1,x_2).
$$
The right-hand side expressions in all three statements of Theorem~\ref{mbu} attain their maxima when $s=\sqrt{Q}=L.$ 
For Parts~(1) and~(3), this is because both $A_{L,f}$ and, respectively, $a_f$ are increasing functions. For Part~(2), we have
$$
8(s-1)\big(4L-s-1\big)\,h(s_0(s))=\frac{8(s-1)(4L-s-1)^3}{9(9L^2-s^2-s-1)^2}\,f(s_0(s)).
$$
It is easy to verify that the fraction in front of $f$ is increasing in $s,$ as is the function $s_0;$ on the other hand, $f$ is increasing by assumption. 

Therefore, we obtain the three statements of Theorem~\ref{ttt1} --- but with ``$\leqslant$'' instead of ``$=$'' in Parts~(1) and~(3) --- by setting $s=\sqrt{Q}$ in the three respective statements of Theorem~\ref{mbu} and changing the variable in the integrals.

To prove that the first and third statements hold with equality, note that by the definition of~$\bel{B}_{Q,\Phi},$ for any $Q\geqslant1,$ any $x_1>0$ and any interval $I,$ there exists a sequence of $A_2^d(I)$-weights $\{w_n\},$ such that for each $n$ $[w_n]_{A_2^d(I)}=Q,$ $\av{w_n}I=x_1,$ $\av{w_n}I\av{w_n^{-1}}I=Q,$ and
$$
\lim_{n\to\infty}\frac1{|I|}\sum_{J\in\mathcal{D}(I)}c_J^\Phi(w_n)=\bel{B}_{Q,\Phi}(x_1,Q/x_1).
$$  
Let us extend each $w_n$ to all of $\mathbb{R}$ periodically by replicating it on each dyadic interval of length $|I|;$ keep the name $w_n$ for the extension. Clearly, $[w_n]_{A^d_2}=[w_n]_{A^d_2(I)},$ which means that the left-hand side does not exceed $K_\Phi(Q).$ On the other hand, by Theorem~\ref{mbu} the right-hand side equals $\A_{L,f}(\sqrt Q)$ in Part~(1) and $a_f(\sqrt Q)$ in Part~(3). The proof is complete.
\end{proof}

\begin{proof}[Proof Theorem~\ref{ttt2}]
We will assume the statements of Theorem~\ref{mbl}, Lemma~\ref{lopt_af}, and Remark~\ref{rem_elem}. The last two come from Section~\ref{optimizers}, which itself is self-contained.
\medskip

\noindent {\it Proof of Part (1).} Take any $Q\geqslant1$ and any $w\in A_2^d$ such that $[w]_{A_2^d}=Q.$ There exists a sequence of intervals $\{I_n\}$ such that $\av{w}{I_n}\av{w^{-1}}{I_n}\to Q$ as $n\to\infty.$ By the definition of $\bel{b}_{Q,\Phi},$
$$
\frac1{|I_n|}\sum_{J\in\mathcal{D}(I_n)}c_J^\Phi(w)\geqslant \bel{b}_{Q,\Phi}(\av{w}{I_n},\av{w^{-1}}{I_n})=a_f\Big(\sqrt{\av{w}{I_n}\av{w^{-1}}{I_n}}\Big).
$$
Therefore,
$$
\|\{c_J^\Phi\}\|_{\mathcal{C}}\geqslant \lim_{n\to\infty} a_f\Big(\sqrt{\av{w}{I_n}\av{w^{-1}}{I_n}}\Big)=a_f\big(\sqrt Q\big),
$$
since $a_f$ is continuous on $[1,\infty).$ Taking the infimum over all $w$ with $[w]_{A_2^d}=Q$ gives
$$
k_\Phi(Q)\geqslant a_f \big(\sqrt Q\big).
$$
To prove the converse inequality, we 
use the sequence $\{w_n^{(L,L)}\}$ of weights on $(0,1)$ constructed in Section~\ref{optimizers}. That sequence is given by~\eqref{opt1.1} and~\eqref{opt1.2} with $s_0=L;$ let us call it simply $\{w_n\}$ here. Let $I=(0,1).$ Lemma~\ref{lopt_af} asserts that $[w_n]_{A_2^d(I)}=Q$ for all $n$ and that 
$$
\lim_{n\to\infty}\sum_{J\in\mathcal{D}(I)}c_J^\Phi(w_n)=a_f\big(\sqrt Q\big).
$$
In addition, from Remark~\ref{rem_elem}, 
$$
\sup_{R\in\mathcal{D}(I)}\frac1{|R|}\sum_{J\in\mathcal{D}(R)}c_J^\Phi(w_n)\leqslant \sum_{J\in\mathcal{D}(I)}c_J^\Phi(w_n).
$$
Extend each $w_n$ periodically to all of $\mathbb{R}.$ Clearly, for each $n,$ $[w_n]_{A_2^d}=Q$ and 
$$
\|\{c_J^\Phi(w_n)\}\|_{\mathcal{C}}=\sum_{J\in\mathcal{D}(I)}c_J^\Phi(w_n).
$$
The left-hand side is no less than $k_\Phi(Q),$ while the right-hand side converges to $a_f(\sqrt Q).$ After changing the variable in the integral $\int_1^Lf(z)/z\,dz,$ we obtain the statement in Part~(1) of the theorem.
\medskip

\noindent {\it Proof of Part (2).} The proof proceeds along the lines of that of Part~(1), but also uses key ingredients from the proof of Theorem~\ref{mbl}(2). First, take $Q\geqslant 1$ and any $w\in A_2^d$ such that $[w]_{A_2^d}=Q.$ There exists a sequence of dyadic intervals $\{I_n\}$ such that $\av{w}{I_n}\av{w^{-1}}{I_n}\to Q.$  Now, as in~\eqref{key},
\begin{align*}
\|\{c_J^\Phi(w)\}\|_{\mathcal{C}}&\geqslant\frac1{|I_n|}\sum_{J\in\mathcal{D}(I_n)}c_J^\Phi(w)=\frac1{|I_n|}\sum_{J\in\mathcal{D}(I_n)} |J|f(s_J(w))R_J(w)\\
&\geqslant\frac{f(L)}{L^2}\,\frac1{|I_n|}\sum_{J\in\mathcal{D}(I_n)} |J|s_J^2(w)\,R_J(w)
\geqslant \frac{f(L)}{L^2}\,\bel{b}_{Q,1}\big(\av{w}{I_n},\av{w^{-1}}{I_n}\big)\\
&=8\frac{f(L)}{L^2}\,\Big(\av{w}{I_n}\av{w^{-1}}{I_n}-1\Big)\xrightarrow[n\to\infty]{} 8\Phi(Q)\Big(1-\frac1Q\Big).
\end{align*}
Therefore, $k_\Phi(Q)\geqslant8\Phi(Q)\big(1-\frac1Q\big).$ To establish the converse, take an interval $I$ and let $w^*$ be the weight on $I$ defined by~\eqref{w**}. Extend $w^*$ to all of $\mathbb{R}$ as in Part~(1). Then $[w^*]_{A_2^d}=Q$ and 
$$
\|\{c_J^\Phi(w)\}\|_{\mathcal{C}}=\frac1{|I|}\sum_{J\in\mathcal{D}(I)}c_J^\Phi(w^*)=8\Phi(Q)\Big(1-\frac1Q\Big).
$$
The left-hand side is no smaller than $k_\Phi(Q),$ hence the proof is complete.
\end{proof}

\begin{proof}[Proof of Theorem~\ref{corr1}]
We first obtain estimates on $K_\Phi$ and $k_\Phi$ that work for all increasing $\Phi.$ Assume that $w\in A_2^d$ and let $[w]_{A_2^d}=Q.$ For all $s\in[1,L],$ 
$f(s)\,\frac{s^2}{L^2}  \leqslant f(s) \leqslant f(L).$ Therefore, for $c_J^\Phi(w)=|J|f(s_J(w))R_J(w)$ we get
$$
c_J^{\Phi_-}(w)\leqslant c_J^\Phi(w) \leqslant c_J^{\Phi_+}(w),
$$
where we have set $\Phi_-(t)=\Phi(t)\,\frac tQ$ and $\Phi_+(t)=\Phi(Q).$ Hence,
$$
k_{\Phi_-}(Q)\leqslant k_\Phi(Q),\qquad K_\Phi(Q) \leqslant K_{\Phi_+}(Q).
$$ 
The function $\Phi_-(t)/t=\Phi(t)/Q$ is increasing, so by Theorem~\ref{ttt2}(1)
$$
k_{\Phi_-}(Q)=8\,\int_1^Q\frac{\Phi_-(t)}t\,dt=\frac8Q\,\int_1^Q\Phi(t)\,dt.
$$
On the other hand, $\Phi_+$ is decreasing, so by Theorem~\ref{ttt1}(3)
$$
K_{\Phi_+}(Q)=8\int_1^Q \frac{\Phi_+(t)}t\,dt=8\Phi(Q)\,\log Q.
$$

This already proves part~(i). To prove (ii), we observe that if $\Phi(t)\to\infty$ as $t\to\infty$ (write $\Phi(\infty)=\infty$), then the function $k_{\Phi_-}$ defined above is strictly increasing and $k_{\Phi_-}(\infty)=\infty.$ Therefore, $k_{\Phi_-}$ is invertible on $[1,\infty).$

Now, assume that $\{c^\Phi_J(w)\}$ is Carleson, but $w\notin A_2^d.$ Take $I\in\mathcal{D}$ and let $w^I=w\,\chi^{}_I.$ For $n\geqslant0,$ let $w^I_n$ be the truncation of $w^I$ at the $n$-th dyadic generation:
$$
w^I_n=\sum_{J\in\mathcal{D}_n(I)}\av{w}J\chi^{}_J.
$$
Extend $w_n^I$ periodically to all of $\mathbb{R};$ still call it $w_n^I.$ Observe that $w_n^I\in A^d_2$ and $\|c^{\Phi}_J(w^I_n)\|_{\mathcal{C}}\leqslant \|c^{\Phi}_J(w)\|_{\mathcal{C}}.$ Therefore, 
\eq[cont]{
k_{\Phi_-}\big([w^I_n]_{A^d_2}\big)\leqslant \|c^{\Phi}_J(w)\|_{\mathcal{C}}.
}

Since $w\notin A^d_2,$ $\sup_{I\in\mathcal{D}}\big(\lim_{n\to\infty}[w^I_n]_{A^d_2}\big)=\infty.$ Since $k_{\Phi_-}(\infty)=\infty,$ if we take first the limit in~\eqref{cont} as $n\to\infty$ and then, if still necessary, the supremum over all $I,$ we obtain a contradiction. Therefore, $w\in A_2^d;$ furthermore, $k_{\Phi_-}([w]_{A_2^d})\leqslant \|\{c_J^\Phi(w)\}\|_{\mathcal{C}}.$ Inverting $k_{\Phi_-}$ (called $u$ in the statement of the theorem) completes the proof of part~(ii).

To prove~(iii), first observe that if $\Phi(\infty)\ne\infty,$ then
$$
\|\{c_J^\Phi(w)\}\|_{\mathcal{C}}\leqslant \Phi(\infty)\,\|\{|J|R_J(w)\}\|_{\mathcal{C}}.
$$
We will now present an explicit weight $w\notin A_2^d$ such that the sequence $\{|J|R_J(w)\}$ is Carleson. For $t\in(0,1),$ let
$$
w(t)=\sum_{k=1}^\infty 2^k k^{-2}\cdot \chi^{}_{(2^{-k},2^{-k+1})}(t).
$$
Now, extend $w$ periodically to $\mathbb{R}.$ For $n\geqslant0,$ let $I_n=(0,2^{-n}).$ We compute 
$$
\av{w}{I_n}=2^n\sum_{k=n+1}^\infty k^{-2}\geqslant 2^n\,n^{-1},\qquad \av{w^{-1}}{I_n}=2^n\sum_{k=n+1}^\infty k^2 4^{-k}\geqslant \frac13\,n^2\,2^{-n}.
$$
Therefore, $\av{w}{I_n}\av{w^{-1}}{I_n}\geqslant n/3,$ and so $w\notin A^d_2.$ 

Now, clearly
$$
\|\{|J|R_J(w)\}\|_{\mathcal{C}}=\sup_{I\in\mathcal{D}(0,1)}\frac1{|I|}\sum_{J\in\mathcal{D}(I)} |J|R_J(w).
$$
Observe, that the only intervals in $\mathcal{D}(0,1)$ on which $w$ is not constant --- and, thus, on which $\sum_{J\in\mathcal{D}(I)}|J|R_J(w)\ne0$ --- are the intervals $I_n$ for $n\geqslant0.$ For each $n$ we have
$$
\frac1{|I_n|}\sum_{J\in\mathcal{D}(I_n)} |J|R_J(w)=\frac1{|I_n|}\sum_{k=n}^\infty |I_k|R_{I_k}(w).
$$
For any weight $w$ and any interval $J,$ $\av{w}{J^\pm}\leqslant 2\av{w}J,$ so $R_J(w)=\big(\av{w}{J^-}-\av{w}{J^+}\big)^2/\av{w}{J}^2+\big(\av{w^{-1}}{J^-}-\av{w^{-1}}{J^+}\big)^2/\av{w^{-1}}{J}^2\leqslant8.$
We conclude that
$$
\|\{|J|R_J(w)\}\|_{\mathcal{C}}=\sup_n\frac1{|I_n|}\sum_{J\in\mathcal{D}(I_n)} |J|R_J(w)\leqslant 8\cdot 2^n\sum_{k=n}^\infty2^{-k}=16,
$$
finishing the proof.
 \end{proof}

\begin{proof}[Proof of Corollary~\ref{corr2}] The only part of this corollary not contained in Theorem~\ref{corr2} is that for $\Phi(t)=t^\alpha,$ Corollary~\ref{ttt3} gives exact expressions for $k_\alpha$ and we do not need the inequality $k_\Phi(Q)\geqslant \frac8Q\int_1^Q\Phi(t)\,dt.$ Observe that $k_\alpha$ is indeed invertible for all $\alpha>0.$
\end{proof}

\begin{proof}[Proof of Corollary~\ref{corr3}]
The statement of the corollary is an immediate consequence of the following lemma. This lemma extends the results of Beznosova~\cite{Bez}, who proved it with $\gamma=1$ and $4$ in place of $e,$ and Moraes, who proved it in his thesis (also see~\cite{bmp}) for all $\gamma\ge0,$ but still with the constant 4. The proof given works in any dimension, though we need it only in dimension~1.

\begin{lemma}
\label{leb}
Let $w$ be a weight such that $w^{-1}$ is also a weight. Assume $\{c_J\}_{J\in \mathcal{D}}$ is a Carleson sequence with norm $B.$ Then for any $\gamma\geqslant 0$ and each $I\in\mathcal{D},$ 
$$
\frac1{|I|}\sum_{J\in\mathcal{D}(I)} c_J\,\frac1{\av{w^{-1}}J^\gamma}\leqslant e\,B\av{w^\gamma}I.
$$
\end{lemma}
\begin{proof}
The case $\gamma=0$ is obvious, so assume $\gamma>0.$ Observe that for any interval $I,$ any $J\in D(I),$ and any $p>1/\gamma,$
$$
\av{w^{-1}}J^{-\gamma}\leqslant\av{w^{1/p}}J^{\,p\gamma}\leqslant \inf_{x\in J} [M(w^{1/p}\chi^{}_I)]^{p\gamma}(x).
$$
Here $M$ is the dyadic maximal function, $M\varphi(x)\df \sup_{x\in R\in\mathcal{D}}\av{|\varphi|}R.$ Recall that for $p>1$ the norm of $M$ on $L^p(\mathbb{R}^n)$ equals $p/(p-1).$

Now, using Carleson Lemma~\ref{carleson} with $F(x)=[M(w^{1/p}\chi^{}_I)]^{p\gamma}$ gives
\begin{align*}
\frac1{|I|}\sum_{J\in D(I)}c_J\, \frac1{\av{w^{-1}}J^\gamma}&\leqslant \frac1{|I|}\sum_{J\in D(I)}c_J\, \inf_{x\in J} [M(w^{1/p}\chi^{}_I)]^{p\gamma}(x)
\leqslant B \frac1{|I|} \int_{\mathbb R}[M(w^{1/p}\chi^{}_I)]^{p\gamma}(x)\,dx\\
&\leqslant B\,\Big(\frac {p\gamma}{p\gamma-1}\Big)^{p\gamma}\frac1{|I|} \|w^{1/p}\chi^{}_I\|^{p\gamma}_{L^{p\gamma}(\mathbb{R})} = B\,\Big(\frac{p\gamma}{p\gamma-1}\Big)^{p\gamma}\,\av{w^\gamma}I.
\end{align*}
Taking the limit as $p\to\infty,$ we obtain the statement of the lemma.
\end{proof}
To prove the corollary, write $\gamma=\alpha-\beta.$ Then
$$
c^{(\alpha,\beta)}_J(w)=|J|\,\av{w}J^\alpha\,\av{w^{-1}}J^\beta\, R_J(w)=|J|\,\big(\av{w}J\,\av{w^{-1}}J\big)^\alpha \av{w^{-1}}J^\gamma\, R_J(w)
=c^{(\alpha)}_J(w)\,\av{w^{-1}}J^{-\gamma}.
$$
Since $\{c^{(\alpha)}_J(w)\}$ is a Carleson sequence with norm not exceeding $K_\alpha([w]_{A^d_2}),$ the result follows from Lemma~\ref{leb}.
\end{proof}

\section{Necessary conditions on Bellman candidates; induction on scales}
\label{induction}
When seeking the Bellman function for an inequality, one typically determines key conditions that this function must satisfy directly from its definition and then derives (or otherwise presents) a {\it candidate} function that possess these properties. The candidate is then shown to equal the true Bellman function -- or at least bound it from above or below, as appropriate. The actual proof of the inequality consists of revealing the exact relationship between the Bellman candidate and the Bellman function and, technically, does not require one to know where the candidate comes from. However, the process of constructing a candidate both illustrates the main inductive component of the proof and is intrinsically tied with the construction of optimizers in the original inequality. Accordingly, we present it in some detail.

The following lemma lists three key properties of $\bel{B}_{Q,\Phi}$ and $\bel{b}_{Q,\Phi}$ that follow from definitions~\eqref{6.6n} and~\eqref{6.7n}.

\begin{lemma}
\label{main_le}
The functions $\bel{B}_{Q,\Phi}$ and $\bel{b}_{Q,\Phi}$ satisfy
\ben
\item{\bf Main inequality.} For all points $x^-,x^+\in\Omega_Q$ such that $(x^-+x^+)/2\in\Omega_Q,$ 
\begin{align}
\label{a1}
\bel{B}_{Q,\Phi}\Big(\frac{x^-+x^+}2\Big)&\geqslant\frac12\bel{B}_{Q,\Phi}(x^-)+\frac12\bel{B}_{Q,\Phi}(x^+)+\Phi(x_1x_2)
\left[\frac{(x_1^--x_1^+)^2}{x_1^2}+\frac{(x_2^--x_2^+)^2}{x_2^2}\right],\medskip\\
\label{a2}
\bel{b}_{Q,\Phi}\Big(\frac{x^-+x^+}2\Big)&\leqslant\frac12\bel{b}_{Q,\Phi}(x^-)+\frac12\bel{b}_{Q,\Phi}(x^+)+\Phi(x_1x_2)
\left[\frac{(x_1^--x_1^+)^2}{x_1^2}+\frac{(x_2^--x_2^+)^2}{x_2^2}\right].
\end{align}
\item{\bf Boundary condition.} For all $x_1\in(0,\infty),$
\eq[bcond]{
\bel{B}_{Q,\Phi}(x_1,1/x_1)=0,\quad \bel{b}_{Q,\Phi}(x_1,1/x_1)=0.
}
\item{\bf Homogeneity.}
For all $x\in\Omega_Q,$
\eq[hom]{
\bel{B}_{Q,\Phi}(x_1,x_2)=\bel{B}_{Q,\Phi}(\sqrt{x_1x_2},\sqrt{x_1x_2}),\quad \bel{b}_{Q,\Phi}(x_1,x_2)=\bel{b}_{Q,\Phi}(\sqrt{x_1x_2},\sqrt{x_1x_2}).
}
\een
\end{lemma}
\begin{proof}
It suffices to prove these statements for $\bel{B}_{Q,\Phi}$ as the proof for $\bel{b}_{Q,\Phi}$ is almost identical.  

To prove~\eqref{a1}, first observe that for any point $x=(x_1,x_2)\in\Omega_Q,$ there exists a weight $w\in \AQ(I)$ such that $\av{w}I=x_1$ and $\av{w^{-1}}I=x_2.$ For instance, one can take
\eq[w*]{
w(t)=\begin{cases}
x_1\Big(1-\sqrt{1-\frac1{x_1x_2}}\Big),&t\in I^-,
\medskip

\\
x_1\Big(1+\sqrt{1-\frac1{x_1x_2}}\Big),&t\in I^+.
\end{cases}
}
Indeed, being constant on $I^\pm,$ this weight is in $A^d_2(I^\pm)$ with $[w]_{A^d_2(I^\pm)}=1.$ Since $x\in\Omega_Q$ and 
$$
\AQ(I)=\{w: w|_{I^-}\in \AQ(I^-), w|_{I^+}\in\AQ(I^+), \av{w}I\av{w^{-1}}I\leqslant Q\},
$$ 
the conclusion follows. This means that the supremum and infimum in the definitions~\eqref{6.6n} and~\eqref{6.7n} are taken over a non-empty set.

Now, for any $w\in\AQ(I),$ 
\begin{align*}
\frac1{|I|}\sum_{J\in\mathcal{D}(I)}c_J^\Phi(w)&=\frac12\,\frac1{|I^-|}\sum_{J\in\mathcal{D}(I^-)}c_J^\Phi(w)+
\frac12\,\frac1{|I^+|}\sum_{J\in\mathcal{D}(I^+)}c_J^\Phi(w)+f(s_I(w))R_I(w),
\end{align*}
where we have used notation~\eqref{sr}. Inequality~\eqref{a1} now follows from considering any weight $w^*\in\AQ(I)$ that realizes (or almost realizes) the supremum in~\eqref{6.6n} on both $I^-$ and $I^+$ (note that $w^*|_{I^-}$ and $w^*|_{I^+}$ can be chosen independently of each other).

The boundary condition~\eqref{bcond} holds since the only functions $w$ in $A_2^d(I)$ satisfying $\av{w}I=1/\av{w^{-1}}I$ are constants.

Lastly, to show~\eqref{hom}, take any $\tau>0$ and replace $w$ with $\tau w$ in the definitions \eqref{6.6n} and~\eqref{6.7n}. The class $\AQ$ is invariant under this transformation, as is the sum $\frac1{|I|}\sum_{J\in\mathcal{D}(I)}c_J^\Phi(w).$ Therefore,
$$
\bel{B}_{Q,\Phi}(\tau x_1,\tau^{-1} x_2)=\bel{B}_{Q,\Phi}(x_1,x_2),
$$
and setting $\tau=\sqrt{x_2/x_1}$ finishes the proof.
\end{proof}
The central point in Bellman analysis on martingales is that if one has {\it any} function $B$ on $\Omega_Q$ that satisfies~\eqref{a1} and~\eqref{bcond}, that function is automatically a majorant of $\bel{B}_{Q,\Phi};$ a similar conclusion applies to $\bel{b}_{Q,\Phi}$. 
Specifically, we have the following lemma, which implements in our setting the main part of any Bellman-function argument, sometimes referred to as ``Bellman induction'' or ``induction on scales.''
\begin{lemma}
\label{le1}
Let $B$ and $b$ be functions on $\Omega_Q$ satisfying  
\eq[a1.1]{
B\Big(\frac{x^-+x^+}2\Big)\geqslant\frac12B(x^-)+\frac12B(x^+)+
\Phi(x_1x_2)\left[\frac{(x_1^--x_1^+)^2}{x_1^2}+\frac{(x_2^--x_2^+)^2}{x_2^2}\right]
}
and
\eq[a2.1]{
b\Big(\frac{x^-+x^+}2\Big)\leqslant\frac12b(x^-)+\frac12b(x^+)+
\Phi(x_1x_2)\left[\frac{(x_1^--x_1^+)^2}{x_1^2}+\frac{(x_2^--x_2^+)^2}{x_2^2}\right],
}
respectively, for all points $x^-,x^+\in\Omega_Q$ such that $(x^-+x^+)/2\in\Omega_Q,$
as well as
\eq[bcond1]{
B(x_1,1/x_1)=0,\quad b(x_1,1/x_1)=0,~\forall x_1\in(0,\infty).
}
Assume that $\Phi$ is a non-negative, bounded function on $[1,Q].$ Then
$$
B(x)\geqslant\bel{B}_{Q,\Phi}(x),\qquad b(x)\leqslant\bel{b}_{Q,\Phi}(x),\qquad\forall x\in\Omega_Q.
$$
\end{lemma}
\begin{proof}
First, we obtain elementary estimates on $B$ and $b.$ Take any $x\in\Omega_Q$ and write $x=(x^-+x^+)/2,$ where, as in~\eqref{w*},
$$
x_1^\pm=x_1\Big(1\pm\sqrt{1-\frac1{x_1x_2}}\Big),\quad x_2^\pm=x_2\Big(1\mp\sqrt{1-\frac1{x_1x_2}}\Big).
$$
We have $x_1^-x_2^-=x_1^+x_2^+=1$ and $(x_1^--x_1^+)^2/x_1^2+(x_2^--x_2^+)^2/x_2^2=8\big(1-1/(x_1x_2)\big).$ Let $M=\|\Phi\|_{L^\infty([1,Q])}.$ Using~\eqref{a1.1} and noting that $B(x^-)=B(x^+)=0$ by~\eqref{bcond1}, we obtain
$$
B(x)\geqslant 8\Phi(x_1x_2)\Big(1-\frac1{x_1x_2}\Big)\geqslant0.
$$
Similarly,
$$
b(x)\leqslant 8\Phi(x_1x_2)\Big(1-\frac1{x_1x_2}\Big)\leqslant 8M\Big(1-\frac1{x_1x_2}\Big).
$$
Now, again take any $x\in\Omega_Q$ and any interval $I.$ Let $w\in\AQ(I)$ be any weight such that $\av{w}I=x_1$ and $\av{w^{-1}}I=x_2.$ For any $J\in\mathcal{D}(I),$ let $x_J=\big(\av{w}J,\av{w^{-1}}J\big)$; since $w\in\AQ(I),$ $x_J\in\Omega_Q.$ 
Applying~\eqref{a1.1} repeatedly, we have 
\begin{align*}
B(x)&\geqslant \frac12 B(x_{I^-})+\frac12 B(x_{I^+})+\Phi(x_1x_2)\,R_I(w)\\
&\geqslant\frac1{|I|}\sum_{J\in\mathcal{D}_n(I)}|J| B(x_J)+\sum_{k=0}^{n-1}\sum_{J\in\mathcal{D}_n(I)}f(s_J)R_J(w).
\end{align*}
Since $B\geqslant0,$ we can drop the first sum and then take the limit as $n\to\infty.$ Taking the supremum over all $w\in\AQ(I)$ with $\av{w}I=x_1$ and $\av{w^{-1}}I=x_2$ proves the lemma for $B.$

The argument for $b$ is similar. Using~\eqref{a2.1} repeatedly, we get
\eq[lb]{
b(x)\leqslant\frac1{|I|}\sum_{J\in\mathcal{D}_n(I)}|J| b(x_J)+\sum_{k=0}^{n-1}\sum_{J\in\mathcal{D}_n(I)}f(s_J)R_J(w).
}
The first sum is bounded by $\frac{8M}{|I|}\sum_{J\in\mathcal{D}_n(I)}|J| \Big(1-\frac1{s_J^2}\Big),$ which converges to $0$ as $n\to\infty$ by the Lebesgue differentiation theorem and the Lebesgue dominated convergence theorem. Taking the limit as $n\to\infty$ and then the infimum over all appropriate $w$ completes the proof.
\end{proof}
\begin{remark}
The reader may wonder if the condition that $\Phi$ be bounded is necessary in this lemma. It is clearly not needed for the conclusion about $B.$ For the lower candidate $b,$ however, some technical assumption is needed in order to disregard the first sum in the right-hand side of~\eqref{lb}. An alternative would be to assume that $b$ is continuous on $\Omega_Q,$ as one can then run the induction only for dyadically-simple weights (for which the induction is finite) and then approximate an arbitrary weight by dyadically-simple ones. In any case, this distinction is inconsequential: first, in all cases of Theorem~\ref{mbl} $\Phi$ is automatically bounded on any finite interval; and second, our only lower candidate, $b(x)=a_f(\sqrt{x_1x_2}),$ is continuous on $\Omega_Q.$
\end{remark}

We now turn to finding  candidates $B$ and $b$ satisfying the hypotheses of Lemma~\ref{le1}. Let us focus on the upper candidate $B;$ for the lower candidate, the steps below are exactly the same, except all inequality signs are reversed.

If $B$ is sufficiently differentiable, condition~\eqref{a1.1} yields the following differential inequality:
\begin{align}
\label{b1}
\frac18\,[dx_1~dx_2]\,{\rm d}^2 B(x)\,[dx_1~dx_2]^T&+\Phi(x_1x_2)\left[\Big(\frac{dx_1}{x_1}\Big)^2+\Big(\frac{dx_2}{x_2}\Big)^2\right]\leqslant0,
\end{align}
where ${\rm d}^2B$ is the Hessian of $B.$

In matrix form, we get
\eq[c1]{
\left[
\begin{array}{ll}
B_{x_1x_1}+8\frac{\Phi(x_1x_2)}{x_1^2}& B_{x_1x_2}\\
\\
B_{x_1x_2}& B_{x_2x_2}+8\frac{\Phi(x_1x_2)}{x_2^2}
\end{array}
\right]\leqslant0.
}
The best (true) candidate $B$ must accommodate the existence of an optimizing weight $w$ (or a sequence thereof), which would produce an equality (or approximate equality) on every step of the inductive process of Lemma~\ref{le1}. In problems that admit infinitesimal forms such as~\eqref{c1} this typically means that the kernel of the corresponding differential matrix is non-trivial. Imposing this condition on $B,$ we obtain the differential equation
\eq[c2]{
\Big(B_{x_1x_1}+8\frac{\Phi(x_1x_2)}{x_1^2}\Big)\Big(B_{x_2x_2}+8\frac{\Phi(x_1x_2)}{x_2^2}\Big)=B^2_{x_1x_2},\quad x\in\Omega_Q.
}
We couple this equation with the boundary condition
\eq[c2.1]{
B(x_1,1/x_1)=0
}
and, to ensure that ${\rm d}^2B\leqslant0,$ with the inequalities
\eq[c3]{
B_{x_1x_1}+8\frac{\Phi(x_1x_2)}{x_1^2}\leqslant0;\quad  B_{x_2x_2}+8\frac{\Phi(x_1x_2)}{x_2^2}\leqslant0.
}

We now solve the system~\eqref{c2},~\eqref{c2.1} using the homogeneity statement~(\ref{hom}) from Lemma~\ref{main_le}. Recall our notation:
$$
s=\sqrt{x_1x_2},\quad f(s)=\Phi(s^2),\quad L=\sqrt Q.
$$ 
Let ${\A}(s)$ stand for either $B(s,s)$ or $b(s,s),$ as needed. With this notation, and upon differentiation,~\eqref{c2} and~\eqref{c2.1} become
\eq[c4]{
\Big(s^2\A''(s)-s\A'(s)+32f(s)\Big)^2=\Big(s^2\A''(s)+s\A'(s)\Big)^2,\quad 1\leqslant s\leqslant L, \quad \A(1)=0.
}
The conditions~\eqref{c3} and their equivalent for $b$ become 
\eq[c5]{
s^2\A''(s)-s\A'(s)+32f(s)\leqslant0\quad\text{and}\quad s^2\A''(s)-s\A'(s)+32f(s)\geqslant0,
}
respectively. Taking square roots in~\eqref{c4}, we obtain two elementary linear equations:
$$
\A'(s)=16\,\frac{f(s)}s;\qquad\qquad \A''(s)=-16\,\frac{f(s)}{s^2}.
$$
Accordingly, we have two solutions, one of which is completely determined by the initial condition $\A(1)=0$ and the other contains a parameter:
\eq[g]{
a_f(s)=16\int_1^s\frac{f(r)}r\,dr
}
and
\eq[G1]{
A_f(s)=16\left[\int_1^s\frac{f(r)}r\,dr-s\int_1^s\frac{f(r)}{r^2}\,dr\right]+C(s-1).
}
How do we determine the constant $C?$ In general, this would depend on whether we intend for $A_f$ to be an upper or lower Bellman candidate, and on what assumptions are made on $f.$ In this paper, we only use $A_f$ as an upper candidate and so we want $C$ to be the smallest number such that the first inequality in~\eqref{c5} is satisfied on 
$[1,L].$ For $\mathcal{A}=A_f$ that inequality is equivalent to
$$
C\geqslant 16\left[\frac{f(s)}s+\int_1^s\frac{f(z)}{z^2}\,dz\right],\qquad 1\leqslant s\leqslant L.
$$
If $f$ is increasing, as it is assumed to be in Part~(1) of Theorem~\ref{mbu}, the maximum of the right-hand side is attained when $s=L,$ thus we set
$$
C= 16\left[\frac{f(L)}L+\int_1^L\frac{f(z)}{z^2}\,dz\right],
$$
which is equivalent to setting $L^2A_f''(L)-LA_f'(L)+32f(L)=0.$ This yields a complete Bellman candidate
\eq[G]{
A_{L,f}(s)=16\left[\frac{f(L)}L+\int_1^L\frac{f(z)}{z^2}\,dz\right](s-1)-16\int_1^s\frac{f(z)}{z^2}\,(s-z)\,dz.
}
In the the next section, we verify that the functions $a_f(\sqrt{x_1x_2})$ and $A_{L,f}(\sqrt{x_1x_2})$ satisfy the appropriate ``main inequalities'' -- either~\eqref{a1.1} or~\eqref{a2.1}, depending on the assumptions on~$f.$

\section{Verification of the main inequality}
\label{main_ineq}
Lemma~\ref{le1} says that if a Bellman candidate satisfies the main inequality,~\eqref{a1.1} or~\eqref{a2.1}, and the boundary condition~\eqref{bcond1}, then it automatically gives an upper or lower estimate on the Bellman function itself. This section is devoted to verifying the main inequality(ies) for the candidates $a_f$ and $A_{L,f}$ (which satisfy the boundary condition by construction) and thus proving the following lemma, which establishes the upper estimates in Parts~(1) and~(3) of Theorem~\ref{mbu} and the lower estimate in Part~(1) of~Theorem~\ref{mbl}. 
\begin{lemma}
\label{lmi}
Let $L=\sqrt{Q},$ $f(z)=\Phi(z^2),$ and $h(z)=f(z)/z^2.$ 
\ben
\item
If $\Phi$ is increasing and $h$ is convex, then
$$
\bel{B}_{Q,\Phi}(x_1,x_2)\leqslant A_{L,f}\big(\sqrt{x_1x_2}\big).
$$

\item
If $\Phi$ is decreasing, then
$$
\bel{B}_{Q,\Phi}(x_1,x_2)\leqslant a_f\big(\sqrt{x_1x_2}\big).
$$
\item
If $h$ is increasing, then
$$
\bel{b}_{Q,\Phi}(x_1,x_2)\geqslant a_f\big(\sqrt{x_1x_2}\big).
$$
\een
\end{lemma}

The proofs of the three separate parts of this lemma are contained in Lemmas~\ref{le0},~\ref{lemka2}, and~\ref{le0.2}. First, let us rephrase~\eqref{a1.1} and~\eqref{a2.1} using the homogeneity of the problem. Suppose $B$ is an upper candidate. Let $\A(s)=B(s,s).$ Alongside $s=\sqrt{x_1x_2},$ we will use $s^\pm=\sqrt{x_1^\pm x_2^\pm}.$ A simple calculation shows that
$$
\frac{(x_1^--x_1^+)^2}{x_1^2}+\frac{(x_2^--x_2^+)^2}{x_2^2}=8-4\,\frac{(s^-)^2+(s^+)^2}{s^2}+\frac{((s^-)^2-(s^+)^2)^2}{s^4},
$$
and so~\eqref{a1.1} is equivalent to
\eq[d1]{
\A(s)-\frac12 \A(s^-)-\frac12 \A(s^+)-f(s)\left[8-4\,\frac{(s^-)^2+(s^+)^2}{s^2}+\frac{((s^-)^2-(s^+)^2)^2}{s^4}\right]\geqslant0.
} 
If $b$ is a lower candidate, and $\A(s)=b(s,s),$ then the inequality to be verified,~\eqref{a2.1}, becomes
\eq[d2]{
\A(s)-\frac12 \A(s^-)-\frac12 \A(s^+)-f(s)\left[8-4\,\frac{(s^-)^2+(s^+)^2}{s^2}+\frac{((s^-)^2-(s^+)^2)^2}{s^4}\right]\leqslant0.
}
It is important to determine the domain of these variables. Clearly $(s^-,s^+,s)\in[1,L]^3,$ but this is too crude. If $s^-$ and $s^+$ are fixed, $s$ cannot be too small; this is reasonably clear from the geometry of the problem. Formally, to find the smallest $s$ we solve the following problem 
$$
\min\left\{\Big(\frac{x_1^-+x_1^+}2\Big)\,\Big(\frac{x_2^-+x_2^+}2\Big):x_1^-x_2^-=(s^-)^2, x_1^+x_2^+=(s^+)^2\right\}.
$$
It is easy to check that this minimum is attained when all three points $x^-, x^+,$ and  $(x^-+x^+)/2$ lie on the same line through the origin. In this case, $s=\sqrt{x_1x_2}=(s^-+s^+)/2.$ Thus the domain over which~\eqref{d1} or~\eqref{d2} needs to be verified for each particular choice of $\A$ is 
\eq[d3]{
\omega_L\df\Big\{(s^-,s^+,s):\quad 1\leqslant s^-\leqslant L;\quad 1\leqslant s^+\leqslant L;\quad \frac{s^-+s^+}2\leqslant s\leqslant L\Big\}.
}

Consider a new function on $\omega_L:$
\eq[p1]{
P(s^-,s^+,s)=\A(s)-\frac12 \A(s^-)-\frac12 \A(s^+)-f(s)\left[8-4\,\frac{(s^-)^2+(s^+)^2}{s^2}+\frac{((s^-)^2-(s^+)^2)^2}{s^4}\right].
}

In the lemmas below, we will set either $\A=A_{L,f}$ or $a_f,$ depending on the conditions assumed on $f$ and on whether we are proving $\A$ to be an upper Bellman candidate, in which case we need to show that $P\geqslant0$ on $\omega_L,$ or a lower one, for which we need to show $P\leqslant0.$ We split further presentation in this section in three parts.

\subsection{$A_{L,f}$ as an upper candidate}
According to the discussion above, we have to show that if $f$ is increasing and $h$ is convex, then $P$ given by~\eqref{p1} with $\A=A_{L,f}$ is non-negative on the domain $\omega_L.$ We will prove a somewhat stronger statement, and one that is much easier to handle computationally. Let 
\eq[r1]{
U(s^-,s^+,s)=A_{L,f}(s)-\frac12 A_{L,f}(s^-)-\frac12 A_{L,f}(s^+)-8f(s)\left[1-\frac{s^-s^+}{s^2}\right].
}
Observe that $P\geqslant U$ on $\omega_L$ and $P\big(s^-,s^+,\frac{s^-+s^+}2\big)=U\big(s^-,s^+,\frac{s^-+s^+}2\big).$
\begin{lemma}
\label{le0}
If $f$ is a non-negative, increasing function on $[1,\infty)$ and $h(z)= f(z^2)/z^2$ is convex, then $U\geqslant0$ on 
$\omega_L.$
\end{lemma}
\begin{proof}
Throughout the proof, we will write $A$ for $A_{L,f}.$ Let us collect a couple of useful facts about $A.$ A direct calculation gives
$$
A'(s)=16\left[\frac{f(L)}L+\int_s^L h(z)\,dz\right].
$$
Since $f$ is increasing, we have
\eq[d1.1]{
A'(s)\geqslant16\,\left[\frac{f(L)}L+f(s)\,\int_s^L\frac1{z^2}\,dz\right]=16\,\frac{f(s)}s.
}
In addition, since $h$ is convex, we have, for any $1\leqslant s_1\leqslant s_2\leqslant L,$
\eq[d1.2]{
A'(s_1)-A'(s_2)=16\int_{s_1}^{s_2} h(z)\,dz\geqslant 16\,(s_2-s_1)\,h\Big(\frac{s_1+s_2}2\Big).
}

In the next two lemmas we first reduce the inequality $U\geqslant0$ on $\omega_L$ to two of its special cases and then verify them.
\begin{lemma}
$U\geqslant0$ on $\omega_L$ if and only if 
\eq[cn1]{
U\Big(s_1,s_2,\frac{s_1+s_2}2\Big)\geqslant0\qquad\text{for all}\quad s_1, s_2\in[1,L] 
}
and
\eq[cn2]{
U(s_1,s_1,s_2)\geqslant0\qquad \text{for all}\quad1\leqslant s_1\leqslant s_2\leqslant L.
}
\end{lemma}
\begin{proof}
Both \eqref{cn1} and~\eqref{cn2} are clearly necessary. To show the sufficiency, take any point $(s^-,s^+,s)\in\omega_L.$ Assume, without loss of generality, that $s^-\leqslant s^+;$ then either $s^-\leqslant s\leqslant s^+$ or $s^-\leqslant s^+<s.$ Let us consider these cases separately.

\subsubsection*{The case $s^-\leqslant s\leqslant s^+$}

Since $s\geqslant (s^-+s^+)/2,$ we have $s^-\leqslant 2s-s^+\leqslant s\leqslant s^+.$
Using~\eqref{d1.1} and~\eqref{d1.2}, we have
$$
A'(s^-)\geqslant A'(2s-s^+)\geqslant A'(s^+)+32(s^+-s)h(s)\geqslant 16\,\frac{f(s^+)}{s^+}+32(s^+-s)h(s).
$$
Differentiating $U$ with respect to $s^-,$ and noting that $f(s^+)\geqslant f(s),$ we get
\begin{align*}
\frac{\partial U}{\partial s^-}&=-\frac12 A'(s^-)+8\,\frac{f(s)}{s^2}\,s^+\\
&\leqslant -8\,\frac{f(s^+)}{s^+}-16(s^+-s)\,\frac{f(s)}{s^2}+8\,\frac{f(s)}{s^2}\,s^+\\
&\leqslant -8\,\frac{f(s)}{s^+}-16(s^+-s)\,\frac{f(s)}{s^2}+8\,\frac{f(s)}{s^2}\,s^+\\
&=-8\,\frac{f(s)}{s^+\,s^2}\,(s^+-s)^2\leqslant0.
\end{align*}
Therefore,
$$
U(s^-,s^+,s)\geqslant U(2s-s^+,s^+,s),
$$
which is positive by~\eqref{cn1}.

\subsubsection*{The case $s^-\leqslant s^+< s$} We have
\begin{align*}
\frac{\partial U}{\partial s^-}&=-\frac12 A'(s^-)+8\,\frac{f(s)}{s^2}\,s^+\\
&\leqslant -\frac12 A'(s)+8\,\frac{f(s)}{s^2}\,s^+\\
&\leqslant -8\,\frac{f(s)}{s}+8\,\frac{f(s)}{s^2}\,s^+\\
&=-8\,\frac{f(s)}{s^2}(s-s^+)\leqslant0.
\end{align*}
Therefore,
$$
U(s^-,s^+,s)\geqslant U(s^+,s^+,s),
$$
which is positive by~\eqref{cn2}.
\end{proof} 
It remains to verify~\eqref{cn1} and~\eqref{cn2}.
\begin{lemma}
Under the assumptions of Lemma~\ref{le0}, inequalities~\eqref{cn1} and~\eqref{cn2} hold.
\end{lemma}
\begin{proof}
To prove~\eqref{cn1}, we need to show that for all $s_1,s_2\in[1,L],$
$$
A\Big(\frac{s_1+s_2}2\Big)-\frac12\,\big(A(s_1)+A(s_2)\big)-2\,h\Big(\frac{s_1+s_2}2\Big)(s_2-s_1)^2\geqslant0.
$$
It is a simple exercise to verify that if $s\in\mathbb{R},$ $\Delta\geqslant0,$  and $u$ is a twice-differentiable function on the interval $[s-\Delta,s+\Delta],$ then
\begin{align*}
u(s)-\frac12\,\big(u(s-\Delta)+u(s+\Delta)\big)&=-\frac12\,\int_{-\Delta}^{\Delta} (\Delta-|t|)u''(s+t)\,dt\\
&=-\frac12\,\int_{-\Delta}^{\Delta} (\Delta-|t|)\Big(\frac12u''(s+t)+\frac12u''(s-t)\Big)\,dt.
\end{align*} 
Using this formula with $u=A,$ $u''=-16h,$ $s=\frac{s_1+s_2}2,$ and $\Delta=|s_2-s_1|/2,$ we have
\begin{align}
\label{int}
A\Big(\frac{s_1+s_2}2\Big)&-\frac12\,\big(A(s_1)+A(s_2)\big)-2\,h\Big(\frac{s_1+s_2}2\Big)(s_2-s_1)^2\\
\notag &=8\int_{-\Delta}^{\Delta} (\Delta-|t|)\Big(\frac12h(s+t)+\frac12h(s-t)-h(s)\Big)\,dt\geqslant0,
\end{align}
where the last inequality follows because $h$ is convex.

To prove~\eqref{cn2}, we have to show that
$$
A(s_2)-A(s_1)\geqslant 8f(s_2)\left[1-\Big(\frac{s_1}{s_2}\Big)^2\right].
$$
Since $A'$ is decreasing and $A'(s)\geqslant 16 f(s)/s,$ we have
$$
A(s_2)-A(s_1)=\int_{s_1}^{s_2}A'(z)\,dz\geqslant A'(s_2)(s_2-s_1)\geqslant 16\,\frac{f(s_2)}{s_2}\,(s_2-s_1)\geqslant 
8f(s_2)\left[1-\Big(\frac{s_1}{s_2}\Big)^2\right]
$$
\end{proof}
The proof of Lemma~\ref{le0} is now complete.
\end{proof}

\subsection{$a_f$ as an upper candidate} 
Here, we have to show that if $f$ is a decreasing function, then $P$ given by~\eqref{p1} with $\A=a_f$ satisfies $P\geqslant0$ on $\omega_L.$ As in the previous case, we instead consider the function
$$
V(s^-,s^+,s)=a_f(s)-\frac12 a_f(s^-)-\frac12 a_f(s^+)-8f(s)\left[1-\frac{s^-s^+}{s^2}\right].
$$
In addition, we will slightly expand the domain. Let
\eq[dom1]{
\omega_\infty=\Big\{(s^-,s^+,s):\quad s^-\geqslant1,\quad s^+\geqslant1,\quad s\geqslant\frac{s^-+s^+}2 \Big\}.
}
Observe that $\omega_L\subset\omega_\infty$ and $P\geqslant V$ on $\omega_\infty.$ 
\begin{lemma}
\label{lemka2}
If $f$ is a non-negative, decreasing function on $[1,\infty),$ then $V\geqslant 0$ on $\omega_\infty.$ 
\end{lemma}
\begin{proof}
The definition of $a_f$ gives
$$
V(s^-,s^+,s)=8\int_{s^-}^s\frac{f(r)}r\,dr+8\int_{s^+}^s\frac{f(r)}r\,dr-8f(s)\left[1-\frac{s^-s^+}{s^2}\right].
$$
Since $f$ is decreasing, we have 
$$
\int_{s^-}^s\frac{f(r)}r\,dr\geqslant f(s)\log\Big(\frac s{s^-}\Big),\qquad \int_{s^+}^s\frac{f(r)}r\,dr\geqslant f(s)\log\Big(\frac s{s^+}\Big)
$$
(note that one of these integrals may be negative, but the inequality still holds), and so
$$
V(s^-,s^+,s)\geqslant 8f(s)\left[\log\Big(\frac{s^2}{s^-s^+}\Big)-1+\frac{s^-s^+}{s^2}\right].
$$
Since $\frac{s^-s^+}{s^2}\leqslant1,$ we conclude that $V(s^-,s^+,s)\geqslant0.$
\end{proof}
\subsection{$a_f$ as a lower candidate}
Here, we show that if $h$ is an increasing function, then the function $P$ given by~\eqref{p1} with $\A=a_f$ satisfies $P\leqslant0$ on $\omega_L.$ We will again prove a slightly stronger result. Namely, let
$$
W(s^-,s^+,s)=a_f(s)-\frac12\,a_f(s^-)-\frac12\,a_f(s^+)-4f(s)\,\left[2-\frac{(s^-)^2+(s^+)^2}{s^2}\right].
$$
Observe that $P\leqslant W$ on the domain $\omega_\infty$ given by~\eqref{dom1} and that $P(s^+,s^+,s)=W(s^+,s^+,s).$
\begin{lemma}
\label{le0.2}
If $f$ is a non-negative function on $[1,\infty)$ such that $h(s)=f(s)/s^2$ is increasing, then $W\leqslant 0$ on $\omega_\infty.$ 
\end{lemma}
\begin{proof}
The proof is similar to that of Lemma~\ref{lemka2}. We have
$$
W(s^-,s^+,s)=8\int_{s^-}^s\frac{f(r)}r\,dr+8\int_{s^+}^s\frac{f(r)}r\,dr-4f(s)\,\left[2-\frac{(s^-)^2+(s^+)^2}{s^2}\right].
$$
Since $h$ is increasing, we can write
$$
\int_{s^-}^s\frac{f(r)}r\,dr=\int_{s^-}^s\frac{f(r)}{r^2}\,r\,dr\leqslant \frac{f(s)}{s^2}\int_{s^-}^s r\,dr=\frac12\,f(s)\Big(1-\frac{(s^-)^2}{s^2}\Big),
$$
and similarly for the second integral, which immediately gives $W(s^-,s^+,s)\leqslant0.$
\end{proof}

\section{The case of concave $h:$ proof of the main estimate and discussion}
\label{concave}
In the case when the function $h(s)=\Phi(s^2)/s^2$ is concave, we do not obtain the exact Bellman functions $\bel{B}_{Q,\Phi}.$ Instead, we derive good pointwise estimates for $\bel{B}_{Q,\Phi}$ using the simple observation that the graph of $h$ lies below all of its tangents (since $h$ is concave, it does have one-sided tangents at every point). Thus, we majorate $h$ by its tangent at a point $s=s_0,$ use Part (1) of Lemma~\ref{lmi} in conjunction with this convex majorant, and then optimize over all $s_0.$

To make presentation smoother, and to illustrate in what sense the resulting estimate can be construed as optimal, the derivation below assumes that $h$ is differentiable. However, the final formula does not use derivatives, but only the fact that for every $s_0\in(1,L],$ there exists a number $m(s_0)$ such that $h(s)\leqslant m(s_0)(s-s_0)+h(s_0).$ Accordingly, we could obtain the same result by replacing all instances of $h'(s_0)$ with $m(s_0)$ and then simply taking the value of $s_0$ given by~\eqref{s_0} below. 

\begin{proof}[Proof of Theorem~\ref{mbu}, Part (2)]
Without loss of generality, assume $L>1.$ Fix $s_0\in(1,L]$ and let $h_{s_0}(s)=h'(s_0)(s-s_0)+h(s_0)$ be the tangent to the graph of $h$ at $s_0.$ Let $f_{s_0}(s)=s^2h_{s_0}(s);$ then $f(s)=s^2h(s)\leqslant f_{s_0}(s)$ or, equivalently, 
$\Phi(s^2)\leqslant\Phi_{s_0}(s^2)\df f_{s_0}(s).$ Therefore, for any $w\in A_2(Q),$
$$
\frac1{|I|}\sum_{J\in \mathcal{D}(I)}c_J^\Phi(w)\leqslant \frac1{|I|}\sum_{J\in \mathcal{D}(I)}c_J^{\Phi_{s_0}}(w),
$$
which means that
$$
\bel{B}_{Q,\Phi}\leqslant \bel{B}_{Q,\Phi_{s_0}}.
$$
Since $h_{s_0}$ is convex, we can estimate the Bellman function on the right using Part~(1) of Lemma~\ref{lmi}:
$$
\bel{B}_{Q,\Phi_{s_0}}(x_1,x_2)\leqslant A_{L,f_{s_0}}\big(\sqrt{x_1x_2}\big),
$$
where $A_{L,f}$ is given by~\eqref{701}. An easy calculation yields
$$
A_{L,f_{s_0}}(s)=\frac83\,(s-1)\left[3\big(h(s_0)-s_0h'(s_0)\big)(4L-s-1)+h'(s_0)(9 L^2-s^2-s-1)\right].
$$
To minimize this expression with respect to $s_0,$ we compute
$$
\frac{\partial A_{L,f_{s_0}}(s)}{\partial s_0}=\frac83\,(s-1)h''(s_0)\left[9 L^2-s^2-s-1-s_0(4L-s-1)\right].
$$
Since $h$ is concave (note that we do not need strict concavity here) and $4L-s-1\geqslant0,$ the minimum is attained at
\eq[s_0]{
s_0(s)=\frac{9L^2-s^2-s-1}{3(4L-s-1)}.
}
It is easy to verify that $s_0\in(1,L].$ Plugging this back into $A_{L,f_{s_0}},$ we get
\eq[hc]{
A_{L,f_{s_0(s)}}(s)=8(s-1)(4L-s-1)\,h(s_0(s)),
}
which completes the proof.
\end{proof}

Let us briefly discuss this result. First, we note that the optimal tangent does not depend on the choice of $h,$ making formula~\eqref{hc} particularly easy to use. This estimate also demonstrates the utility of our sharp estimates for convex $h,$ proved in Lemma~\ref{lmi} of the previous section. As noted earlier, and as is clear from the statement of Corollary~\ref{ttt4}, for power functions $\Phi(t)=t^\alpha,$ $1\leqslant\alpha\leqslant 3/2,$ the estimate $\bel{B}_{Q,f}(x_1,x_2)\leqslant A_{L,f_{s_0}}\big(\sqrt{x_1x_2}\big)$ amounts to writing
$$
f(s)=s^{2\alpha}=(s^2)^{3-2\alpha}(s^3)^{2\alpha-2}
$$
and then using H\"older's inequality to interpolate between the sharp results for $\alpha=1$ and $\alpha=3/2.$ This connection is not apparent from the general formula~\eqref{hc}.

However, the estimate just proved is not sharp; here is one way to improve it. We have
$$
\frac1{|I|}\sum_{J\in \mathcal{D}(I)}c_J^\Phi(w)= \frac1{|I|}\sum_{J\in \mathcal{D}(I)}c_J^{\Phi_{s_0}}(w)
-\frac1{|I|}\sum_{J\in \mathcal{D}(I)}c_J^{\Phi_{s_0}-\Phi}(w).
$$
The function $\tilde{\Phi}\df\Phi_{s_0}-\Phi$ is non-negative, so we can formally write
\eq[mixed]{
\bel{B}^{}_{Q,\Phi}\leqslant \bel{B}{}_{Q,\Phi_{s_0}}-\bel{b}_{Q,\tilde{\Phi}}.
}
Theorem~\ref{mbl} gives the formula for lower Bellman functions $\bel{b}_{Q,\Phi}$ when $h(s)=\Phi(s^2)/s^2$ is increasing on $[1,\infty).$ However, $\tilde{\Phi}(s^2)/s^2$ does not have this property -- in fact, it is decreasing for $s\leqslant s_0$ and increasing for $s\geqslant s_0.$ We could find the Bellman function for $(\Phi-\Phi_{s_0})\chi_{\{s\geqslant s_0\}}$ as a substitute, but, as a willing reader can verify, the resulting formula would be much more cumbersome than~\eqref{hc} and the improvement in the estimate, numerically quite small. In any case, such manipulations cannot produce a sharp estimate, because, as we demonstrate in Section~\ref{optimizers} below, the optimizing weights for the Bellman candidates $A_{L,f}$ and $a_f$ are essentially unique -- and different. This implies that an estimate of the form~\eqref{mixed}, combining sharp estimates that involve both $A_{L,f}$ and $a_f,$ cannot itself be sharp.

If one desires, as we do, to find the actual function $\bel{B}_{Q,\Phi}$ in the case when $h$ is concave, one has to understand the nature of the extremizing split $x=(x^-+x^+)/2$ in the main inequality~\eqref{a1}, i.e., the choice of $x^\pm$ that turns the inequality into an equality or approximate equality. The candidates $A_{L,f}$ and $a_f$ were derived under the assumption that such $x^\pm$ are infinitesimally close to $x$, meaning~\eqref{a1.1} is equivalent to its infinitesimal version,~\eqref{b1}. However, for concave $h,$ $A_{L,f}$ fails the main inequality at every point, as is clear from formula~\eqref{int} in the proof of Lemma~\ref{lmi}. For similar reasons, $a_f$ does not work either. Therefore, the optimizing split is not infinitesimal in this case. 

An approach that seems promising is to consider piecewise-linear concave $h,$ and then approximate an arbitrary concave $h$ by such functions. One then would look for the Bellman candidate $\mathcal{A}$ such that 
$$
\mathcal{A}''(s)=-16h(s)+\sum_k c_k\delta_{s_k}(s),
$$
where $s_k$ are the points of discontinuity of $h',$ $\delta_{s_k}$ are Dirac masses at $s_k,$ and $c_k$ are the negative coefficients chosen so that the main inequality for the resulting candidate $\mathcal{A}$ is satisfied on the whole domain. The presence of $c_k$ would entail a non-infinitesimal optimal split at the point $x=(s_k,s_k).$ The nature of the dependence of the values of $c_k$ on the local behavior of $h$ is the subject of further study.
\section{Optimizers}
\label{optimizers}
In this section we prove the converse inequalities to those established in Lemma~\ref{lmi}, thus completing the proof of Theorems~\ref{mbu} and~\ref{mbl}. This is done through the construction of two optimizing sequences -- one for $a_f$ and one for $A_{L,f}.$ Importantly, these sequences do not depend on $f,$ meaning that the upper Bellman function $\bel{B}_{Q,\Phi}$ is linear with respect to $\Phi$ for all increasing $\Phi$ such that $h(s)=\Phi(s^2)/s^2$ is convex and, separately, for all decreasing $\Phi.$ Likewise, the lower Bellman function $\bel{b}_{Q,\Phi}$ is linear in $\Phi$ on the class of all $\Phi$ such that $h$ is increasing.

Without loss of generality, we will define our optimizers almost everywhere on $I=(0,1).$ Fix a point $x\in\Omega_Q.$ We say that a sequence of functions $\{w^x_n\}$ on $(0,1)$ is an {\it optimizing sequence for a Bellman candidate $\A$ at $x,$} if $\{w^x_n\}$ satisfies the following three conditions:
\begin{align}
\label{con1}
\forall n, &\quad w^x_n\in\AQ((0,1));\\
\label{con2}
\forall n,&\quad  \av{w^x_n}{(0,1)}=x_1, \av{(w^x_n)^{-1}}{(0,1)}=x_2;\\
\label{con3}
\sum_{J\in\mathcal{D}(0,1)}&|J|\Phi\Big(\av{w^x_n}J\av{(w^x_n)^{-1}}J\Big) R_J(w^x_n)\longrightarrow \A\big(\sqrt{x_1x_2}\big),~\text{as}~n\to\infty.
\end{align}
In some settings, we are lucky to have an actual optimizing function, meaning that $w^x_n$ are the same for all $n,$ but it is rare in dyadic problems. The recursive construction of optimizers in this section is similar to the one used in~\cite{vv} and~\cite{ssv}.

Observe that it follows directly from the definitions~\eqref{6.6n} and~\eqref{6.7n} that if conditions~\eqref{con1}-\eqref{con3} are satisfied, then
\begin{align*}
\bel{B}_{Q,\Phi}(x)&\geqslant\mathcal{A}\big(\sqrt{x_1x_2}\big),\\
\bel{b}_{Q,\Phi}(x)&\leqslant\mathcal{A}\big(\sqrt{x_1x_2}\big).
\end{align*}

Note that due to the homogeneity of the problem it suffices to find $\{w^x_n\}$ only on the line $x_1=x_2.$ Indeed, if $\{w^{(s,s)}_n\}$ is an optimizer for the point $(s,s),$ then $\{\tau w_n^{(s,s)}\}$ is an optimizer for the point $(\tau s,\tau^{-1}s),$ and we can simply set $s=\sqrt{x_1x_2}$ and $\tau=\sqrt{x_1/x_2}.$ 

The optimizers for $a_f$ and $A_{L,f}$ are different, but each construction starts with the vector field generated by the kernel of the matrix in~\eqref{c1}. Specifically, replace $B(x_1,x_2)$ in~\eqref{c1} by $\A(\sqrt{x_1x_2}).$ This gives the equivalent matrix  
$$
\left[
\begin{array}{ll}
\ds\frac1{x_1^2}(s^2\A''(s)-s\A'(s)+32f(s))&\ds \A''(s)+\A'(s)\,\frac1s\\
\\
\ds\A''(s)+\A'(s)\,\frac1s&\ds\frac1{x_2^2}(s^2\A''(s)-s\A'(s)+32f(s))
\end{array}
\right],
$$
whose kernel is given by
\eq[o2]{
(s^2\A''(s)-s\A'(s)+32f(s))\,x_2\,dx_1+  (s^2\A''(s)-s\A'(s))\,x_1\,dx_2=0.
}

We now split the presentation in two parts. 

\subsection{The optimizer for $a_f$}
When $\A=a_f,$ we have $\A'(s)=16f(s)/s$ and so~\eqref{o2} becomes
$$
f'(s)(x_2\,dx_1+x_1\,dx_2)=0,
$$
provided $f'$ is defined at $s.$ At points $x$ with $f'(\sqrt{x_1x_2})=0$ or $f'(\sqrt{x_1x_2})$ is undefined,
~\eqref{o2} does not give any information. However, when $f'\ne0,$ we get
\eq[o3]{
\frac{dx_2}{dx_1}=-\frac{x_2}{x_1},
}
meaning that the vector field consists, locally, of tangents to hyperbolas of the form $x_1x_2=C.$

Now, fix $s\in[1,L]$ and $n\in\mathbb{N}.$ For $k=0,...,n,$ let $s_k=\sqrt{s^2(1-k/n)+k/n};$ also let $\Delta_n=\sqrt{(s^2-1)/n}$ so that $s_{k+1}^2=s_k^2-\Delta_n^2.$ We will define $w^{(s,s)}_n$ as a function that is constant on each of the $2^{-n}$ intervals in $\mathcal{D}_n((0,1)).$ In what follows, let us write $w_n^{(k)}$ for $w_n^{(s_k,s_k)};$ in this notation $w_n^x=w_n^{(s,s)}=w_n^{(0)}.$

We start by splitting the point $x=(s,s)$ into $x^-$ and $x^+$ along the tangent vector field~\eqref{o3}: let $x^\pm=(s\pm\Delta_n,s\mp\Delta_n)$ (note: the larger the $n,$ the more infinitesimal the split). Define $w^x_n$ to be the following concatenation of $w^{x^-}_n$ and $w^{x^+}_n:$
$$
w^x_n(t)=\begin{cases}
w^{x^{-}}_n(2t),&t\in(0,1/2),\\
w^{x^{+}}_n(2t-1),&t\in(1/2,1).
\end{cases}
$$
Due to homogeneity, set $w^{x^\pm}_n=\sqrt{\frac{s\pm\Delta_n}{s\mp\Delta_n}}\,w^{(1)}_n;$ equivalently, $w^{x^\pm}_n=\frac{s_1}{s\mp\Delta_n}\,w^{(1)}_n.$ Repeat this process $n-1$ more times, on each step setting
\eq[opt1.1]{
w^{(k)}_n(t)=
\begin{cases}
\ds\frac{s_{k+1}}{s_k+\Delta_n}\,w^{(k+1)}_n(2t),&t\in(0,1/2),
\medskip

\\
\ds\frac{s_{k+1}}{s_k+\Delta_n}\,w^{(k+1)}_n(2t-1),&t\in(1/2,1).
\end{cases}
}
After $n$ steps, we will have $s_n=1$ and, since the only test functions on the boundary are constants, we have to set 
\eq[opt1.2]{
w_n^{(n)}(t)=1,\qquad t\in(0,1).
}
This completely defines $w^x_n.$ The construction is pictured in Figure~\ref{faf}.

\begin{figure}[h]
\centering{
\includegraphics[width=14cm]{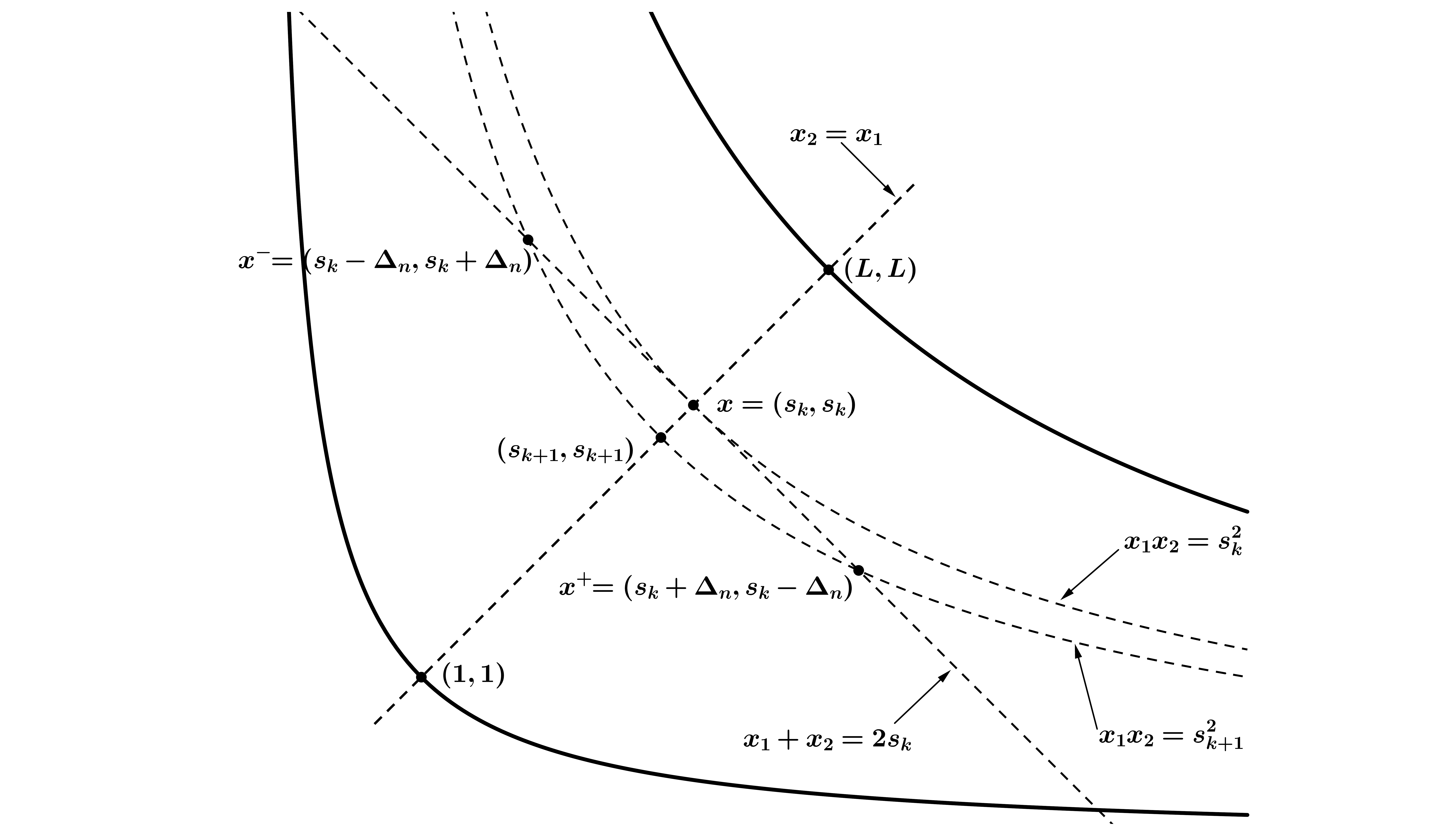}
\caption{The optimizer for $a_f:$ tangential splits $x=\frac12(x^-+x^+)$.}
\label{faf}
}
\end{figure}

Our main result for this section now follows.
\begin{lemma}
\label{lopt_af}
The sequence $\{w_n^{(s,s)}\}$ defined by~\eqref{opt1.1} and~\eqref{opt1.2} is an optimizing sequence for $a_f$ at $(s,s)$ for all $s\in[1,L].$ Consequently,
\ben
\item[(i)]
If $\Phi$ is decreasing, then
$$
\bel{B}_{Q,\Phi}(x)\geqslant a_f\big(\sqrt{x_1x_2}\big),\qquad \forall x\in\Omega_Q.
$$
\item[(ii)]
If $h$ is increasing, then
$$
\bel{b}_{Q,\Phi}(x)\leqslant a_f\big(\sqrt{x_1x_2}\big),\qquad \forall x\in\Omega_Q.
$$
\een
\end{lemma}
\begin{proof}
The verification of~\eqref{con1} and~\eqref{con2} consists of a (backward) finite induction on $k:$ first consider $w^{(n-1)}_n$ and check that $\av{w^{(n-1)}_n}I=s_{n-1},$ $\av{(w^{(n-1)}_n)^{-1}}I=s_{n-1},$ and $w_n^{(n-1)}\in\AQ(I);$ then do the same for $w^{(n-2)}_n,$ and so on. We leave the details to the reader.

It remains to verify~\eqref{con3}. We have
$$
\sum_{J\in\mathcal{D}(0,1)}|J|\Phi\Big(\av{w^x_n}J\av{(w^x_n)^{-1}}J\Big) R_J(w^x_n)=\sum_{k=0}^{n-1}\sum_{J\in\mathcal{D}_k(0,1)}2^{-k}\Phi\Big(\av{w^{x_J}_n}J\av{(w^{x_J}_n)^{-1}}J\Big) R_J(w^{x_J}_n),
$$
where $w^{x_J}_n$ is completely determined by the generation of $J$ and its position among fellow intervals of the same generation. In fact, for any two intervals $I$ and $J$ of the $k$-th generation, $w^{x_I}_n,$ $w^{x_J}_n,$ and $w_n^{(k)}$ differ only by a multiplicative factor, which does not affect their contributions to the overall sum. Thus, 
\begin{align*}
\sum_{k=0}^{n-1}\sum_{J\in\mathcal{D}_k(0,1)}&2^{-k}\Phi\Big(\av{w^{x_J}_n}J\av{(w^{x_J}_n)^{-1}}J\Big) R_J(w^{x_J}_n)\\
&=\sum_{k=0}^{n-1}\sum_{J\in\mathcal{D}_k(0,1)}2^{-k}\Phi\Big(\av{w_n^{(k)}}J\av{(w_n^{(k)})^{-1}}J\Big) R_J(w_n^{(k)})\\
&=\sum_{k=0}^{n-1}\Phi(s_k\cdot s_k)\left[\left(\frac{(s_k-\Delta_n)-(s_k+\Delta_n)}{s_k}\right)^2
+\left(\frac{(s_k+\Delta_n)-(s_k-\Delta_n)}{s_k}\right)^2\right]\\
&=8\sum_{k=0}^{n-1}\Phi(s_k^2)\,\frac{\Delta^2_n}{s_k^2}~~\xrightarrow[n\to\infty]{}~~8\int_1^{s^2}\frac{\Phi(t)}t\,dt=
16\int_1^{s}\frac{f(t)}t\,dt=a_f(s).
\end{align*}

\end{proof}
\begin{remark}
\label{rem_elem}
The optimizer $\{w_n\}$ from~\eqref{opt1.1} and~\eqref{opt1.2} has one more property, which is used in the proof of Theorem~\ref{ttt2} in Section~\ref{elementary}. Namely, for any $R\in\mathcal{D}_k(0,1),$
$$
\frac1{|R|}\,\sum_{J\in\mathcal{D}(R)} c_J^\Phi(w_n)=\sum_{j=0}^{k-1}\Phi(s_j^2)\,\frac{\Delta^2_n}{s_j^2}~\leqslant\!\! \sum_{J\in\mathcal{D}(0,1)} c_J^\Phi(w_n).
$$
Therefore, if one considers the ``local'' Carleson norm of $\{c^\Phi_J(w)\},$ by taking the supremum in~\eqref{2} over all dyadic subintervals of $(0,1),$ as opposed to all dyadic intervals in $\mathcal{D},$ that norm is realized on the the interval $(0,1)$ itself.
\end{remark}
\subsection{The optimizer for $A_{L,f}$}
When $\A=A_{L,f},$ given by~\eqref{G}, we have $s^2\A''(s)-s\A'(s)+32f(s)=-(s^2\A''(s)+s\A'(s))$ and~\eqref{o2} becomes
$$
(s^2\A''(s)+s\A'(s))(x_2\,dx_1-x_1\,dx_2)=0.
$$
A simple calculation shows that $s^2\A''(s)+s\A'(s)=16\Big(f(s)+f(L)\,\frac sL+\int_s^Lh(z)\,dz\Big),$ meaning that if $f$ is monotone and non-constant, then the only point where $s^2\A''(s)+s\A'(s)=0$ is $s=L.$ This gives no information as to the direction in which we should split a point on the boundary. However, infinitesimally, only tangential splitting will keep the new points in the domain, so that will be our choice. At interior points of $\Omega_Q$ we have
\eq[o4]{
\frac{dx_2}{dx_1}=\frac{x_2}{x_1},
}
meaning that the vector field is locally aligned with lines through the origin, $x_2=Cx_1.$

As before, homogeneity allows us to construct an optimizing sequence only along the segment $x_1=x_2=s, 1\leqslant s\leqslant L.$ 
Furthermore, to save effort we will construct an optimizing sequence for the point $x=(s,s)$ only when $s$ is a dyadically-rational point in the interval $[1,L]$ and then use an approximation argument to establish our key result. Thus, let $s=L-m2^{-N}(L-1)$ for some integer $N\geqslant0$ and an odd $m\in\{0,1,...,2^N\}.$

Take $n\geqslant1$ and let $\Delta_n=2^{-N-n}(L-1),$ $s_k=L-k\Delta_n,$ $k=0,1,...,2^{N+n}.$ In this notation, $s=s_{m2^n}.$ Similarly to the previous case, defining an optimizing sequence for the point $(s_{m2^n},s_{m2^n})$ requires defining, for each $n,$ $2^{N+n}+1$ functions, which we will call $w_n^{(k)},$ $k=0,...,2^{N+n}.$ However, in contrast with the optimizer for $a_f,$ the definition of $w_n^{(k)}$ now involves not only $w_n^{(k+1)},$ but also $w_n^{(k-1)}.$

Since $s_{2^{N+n}}=1,$ we set 
\eq[opt2.1]{
w^{(2^{N+n})}_n(t)=1,\qquad t\in(0,1).
}
Assume for the moment that we have already defined the function $w_n^{(0)}$ corresponding to $s_0=L.$ These two ``end-point'' functions are enough for us to derive $w^{(k)}_n$ for all other $k.$ For all $k=1,...,2^{N-n}-1,$ we split the point $(s_k,s_k)$ along the line $x_2=x_1$ (let us call this a {\it normal} split): 
$$
(s_k,s_k)=\frac12((s_{k-1},s_{k-1})+(s_{k+1},s_{k+1}))
$$
and define $w^{(k)}_n$ to be a concatenation of $w^{(k-1)}_n$ and $w^{(k+1)}_n:$
\eq[opt2.2]{
w^{(k)}_n(t)=\begin{cases}
w^{(k-1)}_n(2t),&t\in(0,1/2)
\smallskip

\\
w^{(k+1)}_n(2t-1),&t\in(1/2,1).
\end{cases}
}
The reader can convince himself that this does define each $w^{(k)}_n$ almost everywhere on $(0,1)$ up to the knowledge of $w^{(0)}_n,$ by writing this construction inductively, as done in a similar setting in~\cite{vas}. To find $w_n^{(0)},$ we use the tangetial split from the previous case. As in~\eqref{opt1.1}, define $w_n^{(0)}$ through $w_n^{(1)}$ by
\eq[opt2.3]{
w_n^{(0)}(t)=
\begin{cases}
\ds\frac{s_1}{L+\Delta^*_n}\,w_n^{(1)}(2t),&t\in(0,1/2),
\medskip

\\
\ds\frac{s_1}{L-\Delta^*_n}\,w_n^{(1)}(2t-1),&t\in(1/2,1),
\end{cases}
}
where $\Delta^*_n=\sqrt{L^2-s^2_1}.$ This definition closes the loop: the function $w_n^{(0)}$ has been defined almost everywhere on $(0,1)$ and thus so have the functions $w_n^{(k)}$ for all other $k.$ The construction is pictured in Figure~\ref{falf}. We are now ready to prove our main result.
\begin{figure}[h]
\centering{
\includegraphics[width=14cm]{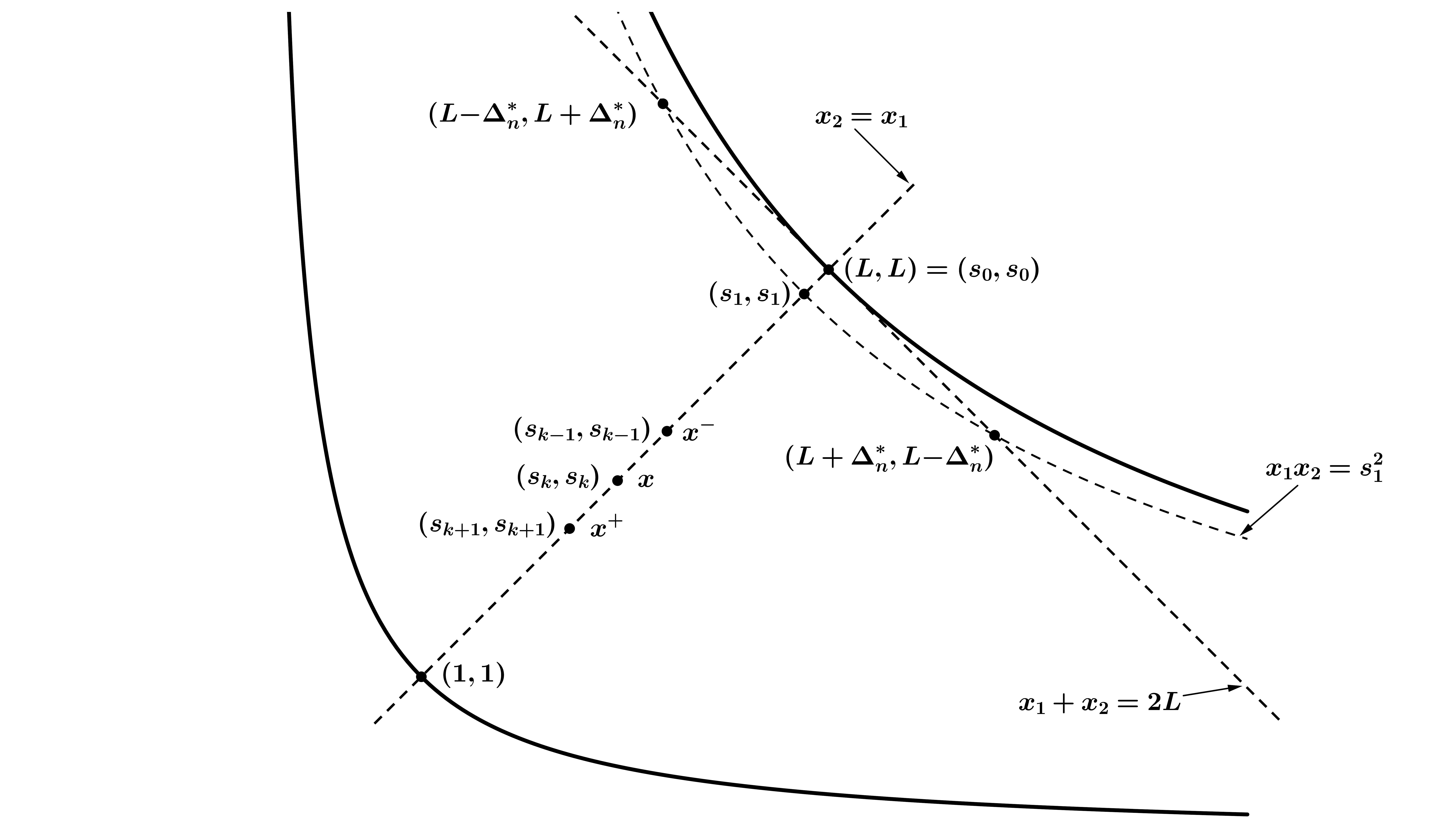}
\caption{The optimizer for $A_{L,f}:$ normal splits $x=\frac12(x^-+x^+)$ inside the domain; tangential split on the boundary.}
\label{falf}
}
\end{figure}

\begin{lemma}
\label{lopt_alf}
The sequence $\{w_n^{(m2^n)}\}$ defined by~\eqref{opt2.1}, \eqref{opt2.2}, and~\eqref{opt2.3} is an optimizing sequence for $A_{L,f}$ at $x=(s,s),$ where $s=L-m2^{-N}(L-1).$ Consequently, if $\Phi$ is increasing and $h$ is concave, then
$$
\bel{B}_{Q,\Phi}(x)\geqslant A_{L,f}(\sqrt{x_1x_2}),\qquad \forall x\in\Omega_Q.
$$
\end{lemma}
\begin{proof}
We will again leave it to the reader to verify~\eqref{con1} and~\eqref{con2}, which in this case is done by induction on the parameter $\ell=n+N.$ Let us show~\eqref{con3}. Let 
$$
\Sigma_k=\sum_{J\in\mathcal{D}(0,1)}|J|\Phi\Big(\av{w_n^{(k)}}J\av{(w_n^{(k)})^{-1}}J\Big) R_J\big(w_n^{(k)}\big).
$$ 
Our recursive construction for $w^{(k)}_n$ directly implies that for $k=1,...,2^{N+n}-1,$
$$
\Sigma_k=\frac12\Sigma_{k-1}+\frac12\Sigma_{k+1}+8h(s_k)\,\Delta_n^2.
$$
On the other hand, $\Sigma_{2^{N+n}}=0,$ while~\eqref{opt2.3} implies $\Sigma_0=\Sigma_1+8(L^2-s_1^2)h(L).$
We can easily solve this tridiagonal linear system:
$$
\Sigma_k=(2^{N+n}-k)8(L^2-s_1^2)h(L)+16\Delta_n^2\left[\sum_{j=1}^{2^{N+n}-1}(2^{N+n}-j)\,h(s_j)-
\sum_{j=1}^k(k-j)\,h(s_j)\right],
$$
for $k=0,...,2^{N+m}$ (where for $k=0$ the second sum in brackets is zero by convention). Since for any $i$ and $j,$  $(i-j)\Delta_n=s_j-s_i,$ and $L^2-s_1^2=\Delta_n(2L-\Delta_n),$ we have
$$
\Sigma_k=8(s_k-1)(2L-\Delta_n)h(L)+16\left[\sum_{j=1}^{2^{N+n}-1}(s_j-1)\,h(s_j)\,\Delta_n-
\sum_{j=1}^k(s_j-s_k)\,h(s_j)\,\Delta_n\right].
$$
Setting $k=m2^n$ and $s_k=s,$ and letting $n\to\infty,$ we get
$$
\Sigma_{m2^n}\xrightarrow[n\to\infty]{}
16(s-1)\frac{f(L)}L+16\left[\int_1^L(z-1)\,h(z)\,dz-\int_s^L(z-s)\,h(z)\,dz\right],
$$
which is the same thing as $A_{L,f}(s)$ from~\eqref{G}.

So far, we have proved that $\bel{B}_{Q,\Phi}(x_1,x_2)\geqslant A_{L,f}\big(\sqrt{x_1x_2}\big)$ for all $(x_1,x_2)\in\Omega_Q$ such that $\sqrt{x_1x_2}$ is a dyadically-rational point in the interval $[1,L].$ Now, any $s\in(1,L)$ can be written as $s=\lim_{n\to\infty}s_n,$ where each $s_n$ is dyadically-rational. Observe that $\bel{B}_{Q,\Phi}(s,s)$ is continuous on $(1,L)$ as a function of $s,$ since it is concave on $[1,L]$ by~\eqref{d1}. Since $A_{L,f}$ is also continuous on $[1,L],$ we can take the limit in the inequality
$$
\bel{B}_{Q,\Phi}(s_n,s_n)\geqslant A_{L,f}(s_n),
$$
thus completing the proof.
\end{proof}

\end{document}